\documentclass[12pt,reqno]{amsart}
\usepackage{amsthm,amsfonts,amssymb,euscript}

\theoremstyle{plain}
  \newtheorem{theorem}[subsection]{Theorem}
  
  \newtheorem{proposition}[subsection]{Proposition}
  \newtheorem{lemma}[subsection]{Lemma}
  \newtheorem{corollary}[subsection]{Corollary}

\def\bphi{{\boldsymbol{\phi} }}
\def\bpsi{{\boldsymbol{\psi} }}
\def\F{{\mathcal F}}
\def\R{{\mathbb{R}}}
\def\I{{\mathcal I}}

\def\A{{\mathcal A}}
\def\B{{\mathcal B}}

\def\Q{{\mathcal Q}}
\def\H{{\mathcal H}}
\def\N{{\mathcal N}}

\def\E{{\mathcal E}}

\def\ch{\mbox{ch} (k)}
\def\chbb{\mbox{chb} (k)}
\def\trigh{\mbox{trigh} (k)}

\def\sh{\mbox{sh} (k)}
\def\shh{\mbox{sh}}
\def\chh{\mbox{ch}}
\def\shhb{\overline {\mbox{sh}}}
\def\th{\mbox{th} (k)}
\def\chhb{\overline {\mbox{  ch}}}
\def\shbb{\mbox{shb} (k)}

\def\shb{\overline{\mbox{sh}(k)}}
\def\z{\zeta}
\def\zb{\overline{\zeta}}
\def\FF{{\mathbb{F}}}
\def\L{\Lambda}

\theoremstyle{remark}
  \newtheorem{remark}[subsection]{Remark}
  
\theoremstyle{definition}

\begin{document}

\def\vac{{\big\vert 0\big>}}
\def\fpsi{{\big\vert\psi\big>}}
\def\bphi{{\boldsymbol{\phi} }}
\def\bpsi{{\boldsymbol{\psi} }}
\def\F{{\mathcal F}}
\def\R{{\mathbb{R}}}
\def\I{{\mathcal I}}
\def\Q{{\mathcal Q}}
\def\ch{\mbox{\rm ch} (k)}
\def\cht{\mbox{\rm ch} (2k)}
\def\chbb{\mbox{\rm chb} (k)}
\def\trigh{\mbox{\rm trigh} (k)}
\def\p{\mbox{p} (k)}
\def\sh{\mbox{\rm sh} (k)}
\def\sht{\mbox{\rm sh} (2k)}
\def\shh{\mbox{\rm sh}}
\def\chh{\mbox{ \rm ch}}
\def\th{\mbox{\rm th} (k)}
\def\thb{\overline{\mbox{\rm th} (k)}}
\def\shhb{\overline {\mbox{\rm sh}}}
\def\chhb{\overline {\mbox{\rm ch}}}
\def\shbb{\mbox{\rm shb} (k)}
\def\chbt{\overline{\mbox{\rm ch}  (2k)}}
\def\shb{\overline{\mbox{\rm sh}(k)}}
\def\shbt{\overline{\mbox{\rm sh}(2k)}}
\def\z{\zeta}
\def\zb{\overline{\zeta}}
\def\FF{{\mathbb{F}}}
\def\S{{\bf S}}
\def\Sr{{\bf S}_{red}}
\def\W{{\bf W}}
\def\WW{{\mathcal W}}
\def\N{{\bf N}}
\def\L{{\mathcal L}}
\def\E{{\mathcal E}}
\def\X{{\tilde X}}
\def\Nor{{\rm Nor}}

\def \kb{{\bf k}}
\def\x{{\bf x}}
\def\fhat{{\widehat{f}}}
\def \ghat{{\widehat{g}}}
\def \fhatb{{\widehat{f}^{\ast}}}
\def \ghatb{{\widehat{g}^{\ast}}}
\def \Sb{\mathbb S}
\def\Zb {\mathbb Z}
\def\Lm{\Lambda}
\def \Lmb{\overline{\Lambda}}
\def \Gam{\Gamma}
\def\Gamb{\overline{\Gamma}}
\def\Lmh{\widehat{\Lambda}}
\def\Fhat{\widehat{F}}
\def\Fhatb{\overline{\widehat{F}}}
\def\E{{\mathcal E}}
\def\Sop{{\bf S}^{t}_{x_{1},x_{2}}}
\def\Wop{{\bf S}^{\pm,t}_{x_{1},x_{2}}}
\def\Sp{{\bf S}^{t}_{x_{1}}}
\def\nb{\nabla}
\def\Dl{\Delta}
\def\Acal{{\mathcal A}}
\def\ub{\overline{u}}
\def\tr{{\rm tr}}
\def\dig{{\rm diag}}

\def\Hbb{{\mathbb H}}
\def\Ubb{{\mathbb U}}
\def\Nbb{{\mathbb N}}

\def\zvec{\vec{z}}
\def\xvec{\vec{x}}
\def\zvecp{\vec{z}^{\prime}}


\title[Estimates for HFB ]%
{ Uniform in $N$ estimates for a Bosonic system of Hartree-Fock-Bogoliubov type}

\author{M. Grillakis}
\address{University of Maryland, College Park}
\email{mng@math.umd.edu}

\author{M. Machedon}
\address{University of Maryland, College Park}
\email{mxm@math.umd.edu}

\date{\today}
\dedicatory{}
\commby{}

\maketitle
\begin{abstract}
We prove local in time, uniform in $N$, estimates for the solutions $\phi$, $\Lambda$ and $\Gamma$ of a coupled system
of Hartree-Fock-Bogoliubov type with interaction potential $v_N(x-y) =N^{3 \beta} v(N^{\beta}(x-y))$, with  $\beta <1$ and $v$ a Schwartz function (satisfying additional technical requirements). The initial conditions are general functions in a Sobolev-type space, and the expected correlations in $\Lambda$ develop dynamically in time.  As shown in our previous work, as well as the work of J. Chong, (both in the case $\beta<2/3$),  using the conserved quantities of the system of equations, this type of local in time estimates can be extended globally. Also, they can be used to derive Fock space estimates for the approximation of the exact evolution of a Bosonic system by quasi-free states of the form $e^{\sqrt N \A(\phi)}e^{\B(k)} \Omega$. This will be addressed in detail in future work.

\end{abstract}
\section{Introduction}

Let $v_N(x-y) =N^{3 \beta} v(N^{\beta}(x-y))$, with  $\beta <1$ and $v$ a Schwartz function satisfying additional technical requirements.
This paper is devoted to obtaining estimates, uniformly in $N$, for solutions to

\begin{subequations}
\begin{align}
& \hskip 0.3 in \left\{\frac{1}{i}\partial_{t}-\Delta_{x_1}\right\}\phi(x_{1}) \label{evphi}
=-\int dy\left\{
v_N(x_{1}-y)\Gamma(y,y)\right\}\phi(x_{1})
\\
&-\int dy\left\{v_N(x_{1}-y)\phi(y)\big(\Gamma(y,x_{1})-\overline{\phi}(y)\phi(x_1)\big)
+v_N(x_{1}-y)\overline{\phi}(y)\big(\Lambda(x_{1},y)-\phi(x_1)\phi(y)\big)\right\} \notag\\
 & \notag\\
&  \hskip 0.3 in \left\{\frac{1}{i}\partial_{t}-\Delta_{x_1}-\Delta_{x_2}+\frac{1}{N}v_N(x_{1}-x_{2})\right\}\Lambda(x_1, x_2) \label{evLambda2}\\
& =-\int dy \left\{v_N(x_{1}-y)\Gamma(y,y) + v_N(x_{2}-y)\Gamma(y,y)\right\}\Lambda(x_1, x_2)
\nonumber
\\
&-\int dy\left\{\big(v_N(x_{1}-y)+v_N(x_{2}-y)\big)\Big(\Lambda(x_{1},y)\Gamma(y,x_{2})+
\overline{\Gamma}(x_{1},y)\Lambda(y,x_{2})\Big)
\right\}+
\nonumber
\\
&+2\int dy\left\{\big(v_N(x_{1}-y)+v_N(x_{2}-y)\big)\vert\phi(y)\vert^{2}\phi(x_{1})\phi(x_{2})\right\}\notag\\
& \notag\\
& \hskip 0.3 in \left\{\frac{1}{i}\partial_{t}-\Delta_{x_1}+\Delta_{x_2}\right\}\overline \Gamma(x_1, x_2)\label{evGamma2}\\
&
=
-\int dy\left\{\big(v_N(x_{1}-y)-v_N(x_{2}-y)\big){\Lambda}(x_{1},y)\overline\Lambda(y,x_{2})\right\}+
\notag
\\
&-\int dy\left\{\big(v_N(x_{1}-y)-v_N(x_{2}-y)\big)\Big(\overline\Gamma(x_{1},y)\overline\Gamma(y,x_{2})+\overline\Gamma(y,y)\overline\Gamma(x_{1},x_{2})\Big)
\right\}
\nonumber
\\
&+2\int dy\left\{\big(v_N(x_{1}-y)-v_N(x_{2}-y)\big)\vert\phi(y)\vert^{2}{\phi}(x_{1})\overline\phi(x_{2})\right\} \notag
\end{align}
\end{subequations}

 The functions $\phi, \Lambda $ and $\Gamma$ also depend  on $N$ and $t$, but this has been suppressed to keep the formulas shorter.
See \ref{normalization} and \ref{BB1}-\ref{BB3},    for the conceptual meaning of these equations, and the normalizations used.
 In particular,  the number of particles is $ N tr (\Gamma)$. As explained in section 2, this is a conserved quantity.
  In addition, the initial conditions are chosen
 to have $O(1)$ Sobolev-type norms.
The spacial dimension is $3$.
These equations are similar in spirit to the Hartree-Fock-Bogoliubov equations for Fermions. For Bosons, they were derived
in \cite{GM2}, \cite{GM3}, and, independently, in \cite{BBCFS}  and the recent paper \cite{BSS}, which addresses both Fermions and Bosons .

Reading this paper requires no knowledge of Fock space. However, as background material, we will review the "big picture" motivating this work, as well as \cite{GMM1}-\cite{GM3}. See these papers for  background on Fock space and additional references related to this project.
Our work (started in \cite{GMM1} in collaboration with D. Margetis) is devoted to the problem of approximating the exact evolution (in Fock space)

\begin{align}
 \psi_{exact}=e^{i t \H}e^{-\sqrt{N}\A(\phi_{0})}e^{-\B(k_{0})}\Omega \label{exact}
 \end{align}
by an expression of the form \ref{approx}.
 Here $\H$ is
 the Fock Hamiltonian
\begin{align}
&\H:=\int dx\left\{
a^{\ast}_{x}\Delta a_{x}\right\} -\frac{1}{2N}\int dxdy\left\{v_{N}(x-y)a^{\ast}_{x}a^{\ast}_{y}a_{x}a_{x}\right\}
\label{FockHamilt1-a}
\end{align}

The function $\phi$ is a function of $3+1$ variables,
\begin{equation}
\A(\phi):=\int dx\left\{\bar{\phi}(x)a_{x}-\phi(x)a^{\ast}_{x}\right\} \label{meanfield-2}
\end{equation}
 and $e^{-\sqrt{N}\A(\phi)}$ is a unitary operator on Fock space,  the Weyl operator.

The "pair excitation" function  $k=k(t, x, y)$ is  symmetric in $x$ and $y$, and
 \begin{align}
&\B(k):=
\frac{1}{2}\int dxdy\left\{\bar{k}(t,x,y)a_{x}a_{y}
-k(t,x,y)a^{\ast}_{x}a^{\ast}_{y}\right\}\ . \label{pairexcit-2}
\end{align}
The unitary operator $e^{\B(k)}$ is the representation of an (infinite dimensional) real symplectic matrix (the Segal-Shale-Weil representation, see
\cite{shale})
 and is called a Bogoliubov transformation in the Physics literature (elements of such a construction go back to \cite{Bog47}).

$\Omega$ is the vacuum, and the state
\begin{equation}\label{cohstates}
\psi:=e^{-\sqrt{N}\A(\phi)}\Omega \ .
\end{equation}
is a coherent state (the $n$th  entry of this state is a tensor product of $n$ copies of $\phi$).
\begin{equation*}
e^{-\sqrt{N}\A(\phi)}\Omega =\left(\ldots\ c_{n}\prod_{j=1}^{n}\phi(x_{j})\ \ldots\right)
\quad {\rm with}\quad  c_{n}=\big(e^{-N \|\phi\|^2_{L^2}}N^{n}/n!\big)^{1/2}\ .
\end{equation*}

 The state
$e^{-\B(k_{0})}\Omega$ is called a squeezed state in the Physics literature.

The Fock space Hamiltonian acts as a PDE Hamiltonian on each entry of Fock space
\begin{align*}
H_{n, \, \, PDE}=\sum_{j=1}^{n}\Delta_{x_{j}} - \frac{1}{N}
\sum_{i<j}N^{3\beta}v\big(N^{\beta}(x_{j}-x_{k})\big)
\end{align*}

In the Math literature, the Fock space evolution $e^{i t \H}$ has been studied in the 70s by by Hepp in \cite{hepp}, Ginibre and Velo \cite{G-V} and, 30 years later, by
Rodnianski and Schlein \cite{Rod-S}, followed by \cite{GMM1} (where $e^{\B}$ is explicitly introduced).

The problem is to find a Fock space approximation of $\psi_{exact}$ (which involves linear Schr\"odinger equation in an unbounded number of variables) by
\begin{equation}\label{approx}
\psi_{approx}:=e^{-\sqrt{N}\A(\phi(t)}e^{-\B(k(t))}\Omega
\end{equation}
where $\phi(t, x)$, $k(t, x, y)$
satisfy non-linear Schr\"odinger equations in $3+1$, respectively $6+1$ variables.
 As explained below, the motivation for the equations discussed in this paper is that they describe a good choice of $\phi$ and $k$ accomplishing this approximation. There are several equivalent ways of writing these equations, see \cite{GM2}, \cite{GM3}.
The equation for $k$ is best expressed in terms of the auxiliary functions $\Gamma$ and $\Lambda$, which are functions of $\phi$ and $k$ and are generalized marginal density matrices for the proposed Fock space approximation.

 Explicitly,
the equations \ref{evphi}-\ref{evGamma2} were written in \cite{GM3}  in an equivalent form, reminiscent of BBGKY, in terms of the marginal densities $\L_{i, j}$, defined by
\begin{align}
&\L_{m, n}(t,y_{1},\ldots, y_{m}; x_{1},\ldots, x_{n}) \label{normalization}\\
&:=\frac{1}{N^{(m+n)/2}}\big<a_{y_1}\cdots a_{y_m} \psi_{approx}, \,  a_{x_1} \cdots a_{x_n}  \psi_{approx} \big> \notag
\end{align}
In terms of these, the equations are

\begin{subequations}
\begin{align}
&\left(\frac{1}{i} \partial_t
-\Delta_{x_1}\right)\label{BB1}
\L_{0, 1}(t, x_1)
  \\
&= -  \int v_N(x_1 - x_2) \L_{1, 2} (t,  x_2; x_1,   x_2) dx_2 \notag \\
& \notag\\
&\left(\frac{1}{i} \partial_t \label{BB2}
+\Delta_{x_1}  - \Delta_{y_1} \right)\L_{1, 1}(t, x_1; y_1)
 \\
&=  \int v_N(x_1 - x_2) \L_{2, 2} (t, x_1, x_2; y_1, x_2) dx_2 \notag
 - \int v_N(y_1 - y_2)\L_{2, 2} (t, x_1, y_2; y_1, y_2)dy_2 \notag\\
 & \notag\\
&\left(\frac{1}{i} \partial_t \label{BB3}
-\Delta_{x_1} -\Delta_{x_2}
+\frac{1}{N}  v_N(x_1-x_2)\right)\L_{0, 2}(t, x_1, x_2)\\
&= -
 \int v_N(x_1-y)\L_{1, 3}(t, y; x_1, x_2, y)dy
  -
 \int v_N(x_2-y)\L_{1, 3}(t, y; x_1, x_2, y)dy \notag
\end{align}
\end{subequations}
Here  (denoting by $\circ$ composition of operators, or the corresponding kernels, $(u \circ v)(x, y)=\int u(x, z)v(z, y) dz$)
\begin{align}
&\phi=\L_{0, 1}\label{def1phi}\\
&\Gamma=\L_{1, 1}=\frac{1}{N}\left(\shb \circ \sh\right)(t, x_1, x_2)+ \bar \phi(t, x_1) \phi(t, x_2)\label{def1Gamma}\\
 &\Lambda=\L_{0, 2}=\frac{1}{2N}\sht (t, x_1, x_2) +\phi(t, x_1) \phi(t, x_2)\label{def1Lambda}
 \end{align}
  where we defined
 \begin{align*}
\sh &:=k + \frac{1}{3!} k \circ \overline k \circ k + \ldots~, \\
\ch &:=\delta(x-y) + \frac{1}{2!}\overline k  \circ k + \ldots~
\end{align*}

The other $\L$ function can be expressed in terms of
$\phi$, $\Lambda$ and $\Gamma$, leading to \ref{evphi}- \ref{evGamma2}.

 Once $\phi$, $\Lambda$ and $\Gamma$ are known, the pair excitation function used in the approximation
\eqref{approx}
can be obtained and estimated from the equation
\begin{align*}
&\S\left(\sht\right)
= - v_N \Lambda\circ \cht
 -\chbt \circ (v_N \Lambda)\\
& -\left(\left(v_N* Tr \Gamma\right)(x)
 + \left(v_N* Tr \Gamma\right)(y)\right) \sht(x, y)\\
&-\left(v_N\Gamma\right)\circ \sht - \sht\circ \left(v_N\Gamma\right)
 \end{align*}
There is a similar equation for $\cht$, see Theorem 7.1 in \cite{GM2}.

The problem of proving a Fock space  estimate for $\|\psi_{exact}-\psi_{approx}\|$ is currently an active field.
See \cite{XC1}, \cite{elif2}, and \cite{BCS}.
 In that paper, Boccato, Cenatiempo and Schlein prove a   result  in the  range $\beta <1$,  and the estimate is global in time.
 Their  approximation
 is given (translating to our notation) by
 $e^{i \chi(t)}e^{-\sqrt{N}\A(\phi(t))}e^{-\B(k(t))}U_{2, N}(t) \Omega$ where
  $\phi$ satisfies the expected cubic nonlinear Schr\"odinger equation,
  $k(t)=k(t, x, y)$ is explicit  and  $U_{2, N}(t)$ is an evolution in Fock space with a quadratic generator (see the page preceding Theorem 1.1 in \cite{BCS}). The function $k(t)$ (which corresponds to pair correlations) has to be present in the initial data.

   Regarding the analysis of such correlations, see also \cite{C-H2}, \cite{EMS}.

  Our approach, based on coupled PDEs for $\phi$ and $k$, offers the possibility to start with uncorrelated initial data, and allow the correlations to develop dynamically,  in very short time.  In the present paper, we obtain estimates for the generalized marginal densities $\Lambda$, $\Gamma$ and $\phi$. In our previous paper \cite{GM3}
  (devoted to the case $0<\beta<\frac{2}{3}$)
  we showed how estimates for $\Lambda$, $\Gamma$ and $\phi$ imply estimates for $k$, which in turn imply Fock space estimates for the proposed approximation \ref{approx}. Our work was extended globally in time in \cite{jacky2}.
  (See also   \cite{jacky1} for the analysis of these equations in $1+1$ dimensions.)

For $\frac{2}{3} \le \beta <1$, there are greater difficulties in both the PDE problem of obtaining estimates for $\Lambda$, $\Gamma$ and $\phi$, and the Fock space approximation. For this reason, we treat the two issues separately. This paper addresses only the PDE problem, leaving the Fock space estimates for future work.  We conjecture  that, using the estimates of this paper, it is possible to prove
\begin{align*}
&\|\psi_{exact}-\psi_{appr}\|_{\F}:=\|e^{i t \H}e^{-\sqrt{N}
\A(\phi_{0})}e^{-\B(k(0))}\Omega-
e^{i \chi(t)}e^{-\sqrt{N}\A(\phi(t))}e^{-\B(k(t))}\Omega\|_{\F}\\
&\le \frac{Ce^{C e^{Ct}}}{N^{\alpha }} \, \, \, \mbox{for } \,  \alpha < \frac{1-\beta}{2}
\end{align*}
for suitable initial conditions, which include pure coherent states ($k_0=0$), and a phase function $\chi$. This type of estimate is comparable to \cite{BCS},
but allows more freedom in the choice of initial conditions.

Fock space techniques can also be applied to $L^2(\mathbb R^N)$ approximations. See the recent paper \cite{BNNS} and the references therein. We also mention the approach of \cite{K-P}.

 We end this section with some general comments regarding the equations studied in this paper. Since we are looking for estimates which hold uniformly in $N$, it is instructive to replace $v_N$ by $\delta$. Then \ref{evphi}-\ref{evGamma2} become, formally,
\begin{align*}
& \hskip 0.3 in \left\{\frac{1}{i}\partial_{t}-\Delta_{x_1}\right\}\phi(x_{1})
=-\Gamma(t, x_1, x_1)\phi(x_{1}) + \cdots
 \\
&  \hskip 0.3 in \left\{\frac{1}{i}\partial_{t}-\Delta_{x_1}-\Delta_{x_2}\right\}\Lambda(x_1, x_2) \\
& =-\left\{\Gamma(x_1,x_1) +
\Gamma(x_2,x_2)\right\}\Lambda(x_1, x_2) + \cdots
\\
& \hskip 0.3 in \left\{\frac{1}{i}\partial_{t}-\Delta_{x_1}+\Delta_{x_2}\right\}\overline \Gamma(x_1, x_2)\\
&
=
-\Lambda(x_{1},x_1)\overline\Lambda(x_1,x_{2})+
 + \Lambda(x_{1},x_2)\overline\Lambda(x_2,x_{2})
\cdots
\end{align*}
 These equations are invariant under the following scaling \\ $\lambda \phi(\lambda^2 t, \lambda x_1)$, $\lambda^2 \Lambda(\lambda^2 t, \lambda x_1, \lambda x_2)$, $\lambda^2 \Gamma(\lambda^2 t, \lambda x_1, \lambda x_2)$, which is $H^1$ critical for $\Lambda$ and $\Gamma$.
However, scaling does not detect the collapse to the diagonal $x_1=x_2$ .
For instance, if we replace $\Lambda(x_{1},x_1)\overline\Lambda(x_1,x_{2})$ by $|\Lambda(x_1,x_{2})|^2$, the scaling would be the same. Since there is loss of regularity in going from $|\Lambda(x_1,x_{2})|^2$ to $|\Lambda(x_1,x_1)|^2$, this make the equations
  $H^1$ supercritical.  On the other hand, as explained in the next section, the conserved energy for the full system has the scaling of $H^2$ for $\Gamma$, but only $H^1$ for $\Lambda$.
(However, as shown in \cite{jacky2}, a norm of $\Lambda$ with the scaling of $H^{1+\epsilon}$ grows in time, but not $N$).
Using these conserved quantities the local existence theorem (and estimates)of the current paper can be extended globally in time. See \cite{jacky2} how this works out the case $\beta<2/3$. We plan to address the problem of global estimates in detail in a forthcoming paper.

The plan of the paper is as follows: Section 2 reviews the conserved quantities (derived in our earlier paper \cite{GM2}), and contains some additional comments. Section 3 has the statement of the main (nonlinear) theorem, and a simple statement of the main linear estimates, which are shown to imply the main non-linear theorem . Sections 4 reviews standard $X^{s, b}$ type estimates needed for the proof. Section 5 introduces new estimates, including a technical statement of the main linear result. The rest of the paper is devoted to the proof of the main linear estimates. Each of those sections starts with a heuristic guide to the proofs in that section.

\bigskip

Acknowledgment. We thank Daniel Tataru for a very useful conversation regarding the estimates of this paper.

\section{Conservation laws, and further comments} The system  \ref{evphi}-\ref{evGamma2}
 has two important conserved quantities.
 We mention this fact here although we do not use it in the  construction of solutions locally in time, where the solutions are obtained by a fixed point argument using the norms and  estimates of Theorem \eqref{mainnonlinthm}.
These conserved quantities would be crucial to extend our estimates globally in time. Solutions to the system considered in this paper, but without potentials depending on $N$, have been constructed in \cite{BBCFS}.

The first conserved quantity is the total number of particles (normalized by division by $N$)
and it is
\begin{align}
M :=
{\rm tr}\left\{\Gamma(t)\right\}\ . \label{b5-nbr0}
\end{align}
This means that $M=1$.
The second conserved quantity is the energy per  particle
\begin{align}
E&:= {\rm tr}\left\{\nabla_{x_{1}}\cdot\nabla_{x_{2}}\Gamma(t)\right\}
+\frac{1}{2}
\int dx_{1}dx_{2}\left\{v_{N}(x_{1}-x_{2})\big\vert\Lm(t,x_{1},x_{2})\big\vert^{2}
\right\}
\nonumber
\\
&+\frac{1}{4}
\int dx_{1}dx_{2}\left\{
v_{N}(x_{1}-x_{2})\left(\big\vert\Gamma(t,x_{1},x_{2})\big\vert^{2}
+\Gamma(t,x_{1},x_{1})\Gamma(t,x_{2},x_{2})\right)
\right\}
\nonumber
\\
&-\frac{1}{2}\int dx_{1}dx_{2}\left\{v_{N}(x_{1}-x_{2})
\vert\phi(t,x_{1})\vert^{2}\vert\phi(t,x_{2}\vert^{2}
\right\}
\label{b5-enrg0}
\end{align}
which is conserved by the evolution. Notice that the quantity $E$ defined above is
positive because
the following inequality is true:
$\Gamma(t,x_{1},x_{2}) \geq \overline{\phi}(t,x_{1})\phi(t,x_{2})$
(in the sense of operator kernels, with $\overline{\phi}(t,x_{1})\phi(t,x_{2})$ representing the kernel of orthogonal projection onto $\bar \phi$.)

In order to see what kind of regularity the conservation laws imply, we observe that the
kinetic part of the energy is,
\begin{align*}
{\rm tr}\left\{\nb_{x_{1}}\cdot\nb_{x_{2}}\Gamma\right\}
=\int dx\left\{\vert\nb_{x}\phi(t,x)\vert^{2}\right\}
+\frac{1}{2N}\int dx_{1}dx_{2}\left\{\vert\nb_{x_{1}}u\vert^{2}+\vert
\nb_{x_{2}}u\vert^{2}\right\}\ .
\end{align*}
Now if we observe that, denoting $u:=\sh$, $c:=\ch$ (see \ref{def1Gamma}, \ref{def1Lambda})
\begin{align*}
\psi :=\sht =2u\circ c\qquad ,\qquad c\circ c=\delta +\ub\circ u
\end{align*}
a simple calculation implies the inequality,
\begin{align*}
&\int dx_{1}dx_{2}\left\{
\big\vert\nb_{x_{1}}\psi\big\vert^{2}+\big\vert\nb_{x_{2}}\psi\big\vert^{2}\right\}
\\
&\leq 4\left(\int dx_{1}dx_{2}\left\{
\vert\nb_{x_{1}}u\vert^{2}+\vert\nb_{x_{2}}u\vert^{2}\right\}\right)
\left(1+\int dx_{1}dx_{2}\vert u\vert^{2}\right)\ .
\end{align*}
Since we have
$\Lm=\phi\phi+(1/2N)\psi$
then
\begin{align*}
\int dx_{1}dx_{2}\left\{
\big\vert\nb_{x_{1}}\Lm\big\vert^{2}+\big\vert\nb_{x_{2}}\Lm\big\vert^{2}\right\}
\leq CE\big(M+N^{-1}\big)
\end{align*}
i.e. the $H^{1}$ norm of $\Lm$ is bounded by a constant which is
 independent of time.
  In fact one can show, see \cite{jacky2}, for some $\epsilon >0$,
 \begin{align*}
 \int dx_{1}dx_{2}\left\{
\big\vert\nb_{x_{1}}\big\vert^{1/2+\epsilon}\big\vert\nb_{x_{2}}\big\vert^{1/2+\epsilon}\Lm\right\}^2
\leq C(t)
\end{align*}
 where $C(t)$ is independent of $N$.

 On the other hand,
$\Gamma =\overline{\phi}\phi+(1/N)\ub\circ u$
which implies a better estimate, namely
\begin{align*}
\big\Vert\vert\nb_{x_{1}}\vert\vert\nb_{x_{2}}\vert\Gamma(t)\Vert_{L^{2}(dx_{1}dx_{2})}
\leq E\ .
\end{align*}
It is interesting to observe that the construction of solutions is accomplished  using norms
which are ( slightly) below the thresholds indicated by the above estimates, see Theorem \ref{mainnonlinthm}.

\begin{remark}
It is an interesting task (at this point) to write down the evolution equations in the case
of aligned Fermions. For the simple Hamiltonian ( defined on Fermionic Fock space)
\begin{align*}
\H:=\int dx_{1}dx_{2}\left\{
c^{\ast}_{x_{1}}\big(\Dl_{x_{1}}\delta(x_{1}-x_{2})\big)c_{x_{2}}
+c^{\ast}_{x_{1}}c^{\ast}_{x_{2}}v(x_{1}-x_{2})c_{x_{2}}c_{x_{1}}
\right\}
\end{align*}
we have the following type of reduction.
We form two (pair) functions $\omega(t,x_{1},x_{2})$ and
$\psi(t,x_{1},x_{2})$. They satisfy certain relations.
Namely $\omega^{\ast}=\omega$ and $\psi^{T}=-\psi$.
Moreover $\omega(t,x,x)\geq 0$ and they are not independent of each other, they
must satisfy
\begin{align*}
-\psi\circ\overline{\psi} +\omega\circ\omega -2\omega =0\ .
\end{align*}

The two evolution equations are,
\begin{align*}
&\frac{1}{i}\partial_{t}\psi +\left(\Dl_{x_{1}}+\Dl_{x_{2}}-2v(x_{1}-x_{2})\right)\psi
\\
&+\int dy\left\{v(x_{1}-y)\left(\overline{\omega}(x_{1},y)\psi(y,x_{2})
+ \psi(x_{1},y)\omega(y,x_{2})-\omega(y,y)\psi(x_{1},x_{2})
\right)
\right\}
\\
&+\int dy\left\{v(y-x_{2})\left(\psi(x_{1},y)\overline{\omega}(y,x_{2})
+\omega(x_{1},y)\psi(y,x_{2})-\omega(y,y)\psi(x_{1},x_{2})
\right)
\right\}=0
\end{align*}
and
\begin{align*}
&-\frac{1}{i}\partial_{t}\omega +\left(-\Dl_{x_{1}}+\Dl_{x_{2}}\right)\omega
\\
&-\int dy\left\{v(x_{1}-y)\left(\overline{\psi}(x_{1},y)\psi(y,x_{2})
+\omega(x_{1},y)\omega(y,x_{2})-\omega(y,y)\omega(x_{1},x_{2})\right)
\right\}
\\
&+\int dy\left\{ v(y-x_{2})\left(\psi(x_{1},y)\overline{\psi}(y,x_{2})
+\omega(x_{1},y)\omega(y,x_{2})-\omega(y,y)\omega(x_{1},x_{2})\right)
\right\}=0\ .
\end{align*}
The analogy is $\Gamma\mapsto \omega^{T}$ and $\Lambda\mapsto \psi$
but notice that $\psi$ is now anti-symmetric!

The system conserves the number of particles,
\begin{align*}
N:=\frac{1}{2}{\rm tr}\big(\omega\big)=\int dx\left\{
\omega(t,x,x)\right\}
\end{align*}
and the total energy,
\begin{align*}
E:=&\frac{1}{2}{\rm tr}\left\{\nabla_{x_{1}}\cdot\nabla_{x_{2}}\omega(t,x_{1},x_{x_{2}})\right\}
+\frac{1}{4}\int dx_{1}dx_{2}\left\{
v(x_{1}-x_{2})\big\vert\psi(t,x_{1},x_{2})\big\vert^{2}\right\}
\\
&+\frac{1}{4}\int dx_{1}dx_{2}\left\{v(x_{1}-x_{2})\left(
\omega(t,x_{1},x_{1})\omega(t,x_{2},x_{2})-\big\vert\omega(t,x_{1},x_{2})
\big\vert^{2}\right)
\right\}\ .
\end{align*}
We may observe that we can set (consistently) $\psi =0$ is which case we obtain the
familiar equation for $\omega$ (which must be a projection).

The equations above are written without any scaling. If we scale space by $N^{\beta}$
then we can rescale appropriately the interaction potential $v$.
The most natural assumption for $v$ is Coulomb. See \cite{BPS}  and \cite{BSS} for related work on Fermionic systems.

\end{remark}

We end this section with some comments on the structure of the nonlinear terms
which (if thought the right way) look "simple".
We adopt the following convention for (skew-Hermitian) commutators and symmetrization
for two operators, say $A$ and $B$
\begin{align*}
\big[A,B\big]:=A\circ B-B^{\ast}\circ A^{\ast}\qquad ,\qquad
\big\{A,B\big\}:=A\circ B+B^{T}\circ A^{T}
\end{align*}
where recall $A\circ B$ means,
\begin{align*}
\big(A\circ B\big)(x_{1},x_{2}):=
\int dy\left\{A(x_{1},y)B(y,x_{2})\right\}\ .
\end{align*}
Also, denote
\begin{align*}
&\Sop=\frac{1}{i}\partial_{t} -\Delta_{x_1}  -\Delta_{x_2}\\
&\Wop=\frac{1}{i}\partial_{t} -\Delta_{x_1}  +\Delta_{x_2}
\end{align*}
With this convention we can write the equations as follows,
\begin{align}
\left\{\Sop +\frac{1}{N}v_{N}\right\}\Lambda
&+\left\{v_{N}\ast{\rm diag}\overline{\Gamma} ,\Lambda\right\}
+\left\{v_{N}\overline{\Gamma},\Lambda\right\}
+\left\{v_{N}\Lambda,\Gamma\right\}
\nonumber
\\
&=\left\{v_{N}\phi\overline{\phi},\phi\phi\right\}
+\left\{v_{N}\phi\phi,\overline{\phi}\phi\right\}
\label{b5-lambda1}
\end{align}
Notice that $\Gamma^{T}=\overline{\Gamma}$ and $\Lambda^{T}=\Lambda$,
and $v_{N}\Lambda$ means pointwise multiplication i.e.
\begin{align*}
&\big(v_{N}\Lambda\big)(x_{1},x_{2}):=v_{N}(x_{1}-x_{2})\Lambda(x_{1},x_{2})
\\
&\big( v_{N}\ast{\rm diag}\Gamma\big)(x_{1},x_{2}) :=
\left(\int dy\left\{v_{N}(x_{1}-y)\Gamma(y,y)\right\}\right)\delta(x_{1}-x_{2})
\ .
\end{align*}
To see why \ref{b5-lambda1} is correct notice that
\begin{align*}
\big\{v_{N}\Lambda ,\Gamma\big\}&=
(v_{N}\Lambda)\circ \Gamma +\Gamma^{T}\circ(v_{N}\Lambda)^{T}
\\
&=\int dy\left\{v_{N}(x_{1}-y)\Lambda(x_{1},y)\Gamma(y,x_{2})
+\Gamma^{T}(x_{1},y)v_{N}(y-x_{2})\Lambda(y,x_{2})\right\}\ .
\end{align*}

The equation for $\Gamma$ reads,
\begin{align}
\Wop\Gamma &+\big[v_{N}\ast{\rm diag}\Gamma,\Gamma\big]+
\big[v_{N}\overline{\Lambda},\Lambda\big]+
\big[v_{N}\Gamma,\Gamma\big]
\nonumber
\\
&=\big[v_{N}\overline{\phi\phi},\phi\phi\big]+
\big[v_{N}\overline{\phi}\phi,\overline{\phi}\phi\big]\ .
\label{b5-gamma1}
\end{align}
Notice that $\Lambda^{\ast}=\overline{\Lambda}$ and $\Gamma^{\ast}=\Gamma$
in computing the commutators and it easy to check that it is indeed correct.

\section{Statement of the main theorem}

We use the standard notation
\begin{align*}
\|F\|_{L^p(dt)L^q(dx) L^2(dy)}=
\bigg\|\big\|F\|_{L^2(dy)}\big\|_{L^{q}(dx)}\bigg\|_{L^{p}(dt)}
\end{align*}

Fix $\alpha>\frac{1}{2}$,  chosen so that $2 \alpha \beta <1$. (This is so that, roughly speaking, $|\nabla|^{2 \alpha} \frac{1}{N} v_N$ is less singular than the delta function.)

Let $0<T<1$ and $c(t)$ the characteristic function of $[0, T]$. Define the norms

  of $\Lambda(t, x, y)$, $\Gamma(t, x, y)$ and $\phi(t, x)$:
\begin{align*}
 N_T(\Lambda)=&\|<\nabla_x>^{\alpha}<\nabla_y>^{\alpha} c(t)\Lambda\|_{L^{2}(dt) L^{6}( dx) L^2(dy)}\\
 &+\|<\nabla_x>^{\alpha}<\nabla_y>^{\alpha}c(t)\Lambda\|_{L^{\infty}(dt) L^{2}( dx) L^2(dy)}\\
&+\mbox{same norm with $x$ and $y$ reversed}\\
&+\sup_w\|<\nabla>^{\alpha} c(t)\Lambda (t, x, x+w)\|_{L^2(dt dx)}\\
&+\sup_w\|\big|\partial_t\big|^{\frac{1}{4}} c(t)\Lambda (t, x, x+w)\|_{L^2(dt dx)}\\
\end{align*}
and, for $\Gamma(t, x, y)$,
\begin{align*}
\dot N_T(\Gamma)= &\|<\nabla_x>^{\alpha}<\nabla_y>^{\alpha} c(t)\Gamma\|_{L^{2}(dt) L^{6}( dx) L^2(dy)}\\
&+\mbox{same norms with $x$ and $y$ reversed}\\
&\|<\nabla_x>^{\alpha}<\nabla_y>^{\alpha} c(t)\Gamma\|_{L^{\infty}(dt) L^{2}( dx) L^2(dy)}\\
&+\sup_w\|<\nabla>^{\alpha-\frac{1}{2}}|\nabla|^{\frac{1}{2}} c(t)\Gamma (t, x, x+w)\|_{L^2(dt dx)}
\end{align*}
and Strichartz norms for $\phi$
\begin{align*}
 N_T(\phi)=&\|<\nabla_x>^{\alpha} c(t)\phi\|_{L^{2}(dt) L^{6}( dx)}
 +\|<\nabla_x>^{\alpha} c(t)\phi\|_{L^{\infty}(dt) L^{2}( dx)}
\end{align*}
\begin{remark} The dot in $\dot N_T(\Gamma)$ is meant to remind the reader that these are not inhomogeneos Sobolev norms.
\end{remark}
\begin{remark} For any Strichartz admissible pair $p, q$ ($\frac{2}{p}+\frac{3}{q}=\frac{3}{2}$) with $2 \le p \le \infty$,
\begin{align*}
\|<\nabla_x>^{\alpha}<\nabla_y>^{\alpha}\Lambda\|_{L^{p}([0, T]) L^{q}( dx) L^2(dy)} \lesssim N_T(\Lambda)
\end{align*}
This can be seen by interpolation. The same comment applies to $\dot N_T(\Gamma)$ and  $N_T(\phi)$.
\end{remark}

Our main (nonlinear) theorem is the following:
\begin{theorem}\label{mainnonlinthm}
Assume $v$ is a Schwartz function with $\hat v$  supported in the unit ball, such that $ | \hat v| \le \hat w$ with $w$ a Schwartz function.   Let $\beta < 1$, and fix $\alpha> \frac{1}{2}$ so that $2 \alpha \beta <1$. Let $\Lambda$, $\Gamma$ and $\phi$ be solutions to \ref{evphi},  \ref{evLambda2}
and  \ref{evGamma2}. Then there exists, $N_0\in \mathbb N$ and $\epsilon >0$ such that, if $0<T<1$ and $N \ge N_0$,
\begin{align*}
&N_T\left(\Lambda\right)
\lesssim
\|<\nabla_x>^{\alpha}<\nabla_y>^{\alpha}\Lambda(0, \cdot)\|_{L^2}\\
&+ T^{\epsilon} \bigg( N_T\left(\Lambda\right)
 \dot N_T\left(\Gamma\right)+ N^4_T\left(\phi\right)\bigg)
\end{align*}
\begin{align*}
&\dot N_T\left(\Gamma\right)
\lesssim
\|<\nabla_x>^{\alpha}<\nabla_y>^{\alpha}\Gamma(0, \cdot)\|_{L^2}\\
&+ T^{\epsilon} \bigg( N^2_T\left(\Lambda\right)+
\dot N^2_T\left(\Gamma\right)+ N^4_T\left(\phi\right)\bigg)
\end{align*}

\begin{align*}
&N_T\left(\phi\right)
\lesssim
\|<\nabla_x>^{\alpha}\phi(0, \cdot)\|_{L^2}\\
&+  T^{\epsilon} \bigg(  N_T\left(\Lambda\right) +  \dot N_T\left(\Gamma\right)
+N^2_T\left(\phi\right)\bigg)N_T\left(\phi\right)
\end{align*}
\end{theorem}
As an immediate consequence, we get
\begin{corollary}
There exists $T_0 >0$ such that,
if $0\le T \le T_0$,
\begin{align*}
&N_T\left(\Lambda\right)+\dot N_T\left(\Gamma\right)+N_T\left(\phi\right)\\
&\lesssim
\|<\nabla_x>^{\alpha}<\nabla_y>^{\alpha}\Lambda(0, \cdot)\|_{L^2}+
\|<\nabla_x>^{\alpha}<\nabla_y>^{\alpha}\Gamma(0, \cdot)\|_{L^2}\\
&+
\|<\nabla_x>^{\alpha}\phi(0, \cdot)\|_{L^2}
\end{align*}
\end{corollary}
\begin{remark}
For the purpose of future applications, we note that
similar estimates hold for the derivatives which commute with the potential:
\begin{align}
&N_T\left(\partial_t   \nabla_{x+y} ^j\Lambda \right) +\dot N_T\left(\partial_t \nabla_{x+y} ^j\Gamma\right)
+N_T\left(\partial_t \nabla_{x} ^j\phi\right)
 \label{commute}\\
&\lesssim \notag
\|<\nabla_x>^{\alpha}<\nabla_y>^{\alpha}\partial_t \nabla_{x+y} ^j
\Lambda \bigg|_{t=0}\|_{L^2}\notag\\
&+\notag
\|<\nabla_x>^{\alpha}<\nabla_y>^{\alpha}
\partial_t \nabla_{x+y} ^j
\Gamma \bigg|_{t=0}\|_{L^2}\\
&+\notag
\|<\nabla_x>^{\alpha}
\partial_t\nabla_{x} ^j
\phi |_{t=0}\|_{L^2}
\end{align}
The time interval $T_0$ and the implicit constants in the above inequalities depend only on the following norm of the initial data:
\begin{align*}
&\|<\nabla_x>^{\alpha}<\nabla_y>^{\alpha}\Lambda(0, \cdot)\|_{L^2}+
&\|<\nabla_x>^{\alpha}<\nabla_y>^{\alpha}\Gamma(0, \cdot)\|_{L^2}\\
&+\|<\nabla_x>^{\alpha} \phi(0, \cdot)\|_{L^2}
\end{align*}
 Also, the theorem is valid for general functions $\phi$, $\Lambda$ and $\Gamma$, although, for the application which motivates this work, they satisfy a constraint (which is preserved by the evolution).
\end{remark}

\begin{remark}
While the above theorem looks similar to Theorem 6.1 in \cite{GM3}, the similarity is superficial. In \cite{GM3} we took $\beta<2/3$ in order to be able to treat $ \frac{1}{N}v_N$ as a perturbation in $X^{\frac{1}{2}}$ type spaces.
The reason for that was $<\nabla_x>^{\alpha}<\nabla_y>^{\alpha} \frac{1}{N}v_N(x-y) \in L^{6/5}$, and we were able to use the $L^2(dt) L^{6/5}(d(x-y))L^2(d(x+y))$ dual Strichartz space.
In the present case,
$\beta<1$, $\alpha > \frac{1}{2}$, this is not the case, since  $<\nabla_x>^{\alpha}<\nabla_y>^{\alpha} \frac{1}{N}v_N(x-y)$
is (almost) as singular as $\delta(x-y)$.

Roughly speaking, $\Lambda $ is the sum of some pieces for which
$<\nabla_x>^{\alpha}<\nabla_y>^{\alpha} \Lambda\in X^{\frac{1}{2}+}$ and one piece for which $<\nabla_x+\nabla_y>^{\alpha}<\nabla_x>^{\alpha}<\nabla_y>^{\alpha}\Lambda \in X^{\frac{1}{4}+}$,
 as well as $<\partial_t>^{\frac{1}{4}} \in X^{\frac{1}{4}+}$
 see the term $D_{z, 2}$ in Lemma \ref{1/4}. These estimates are sufficient in order to recover all Strichartz-type estimates.
\end{remark}
The proof of this theorem follows  from  estimates for solutions to linear equations.
In order to prove these, we have to recall the definition of the spaces $X^{\delta}=X^{s, \delta}$ with $s=0$. See \cite{taobook} for more regarding these.

Denote
\begin{align*}
&\S=\frac{1}{i}\partial_{t} -\Delta_x  -\Delta_y\, \, \\
&\mbox{or, depending on the dimensions,}\\
&\S=\frac{1}{i}\partial_{t} -\Delta_x  \, \, \\
&\S_{\pm}=\frac{1}{i}\partial_{t} -\Delta_x  +\Delta_y\, \, \\
\end{align*}
so that the symbol of $\S$ is  $\tau  +|\xi|^2  +|\eta|^2$
and the symbol of  $\S_{\pm}$ is  $\tau  +|\xi|^2  -|\eta|^2$.
 Let
\begin{align*}
&\|f\|_{X^{\delta}}=\|<\tau +|\xi|^2 +|\eta|^2>^{\delta} \hat f(\tau, \xi, \eta)\|_{L^2(d \tau d \xi d\eta)}\\
&\mbox{or, depending on the dimensions,}\\
&\|f\|_{X^{\delta}}=\|<\tau +|\xi|^2 >^{\delta} \hat f(\tau, \xi)\|_{L^2(d \tau d \xi )}\\
&\|f\|_{X_{\pm}^{\delta}}=\|<\tau +|\xi|^2 -|\eta|^2>^{\delta} \hat f(\tau, \xi, \eta)\|_{L^2(d \tau d \xi d\eta)}\,
\end{align*}
We note that the first application of $X^{\delta}$ spaces to the BBGKY hierarchy is in the work of X. Chen and J. Holmer
\cite{C-H1}.

We solve
\begin{align}
&\S_{\pm} \Gamma= F \label{diffeq}\\
&\Gamma(0)=\Gamma_0 \notag
\end{align}
in an interval$[0, T]$
by constructing
\begin{align*}
&  \Gamma_1=
  e^{ i t \Delta_{\pm}}\Gamma_0 \\
  & +
  \int_{-\infty}^{\infty}c(t-s)e^{ i (t-s) \Delta}c(s) F(s) ds
\end{align*}
where $c(t)$ is the characteristic function of $[0, T]$, $T \le 1$,
and noticing $c(t)\Gamma=c(t)\Gamma_1$ .

The following is known, it follows from Proposition 5.8 in \cite{GM3}, and Lemma \ref{energy}.
See also \cite{CHP} and
\cite{C-H1}.
\begin{proposition}\label{oldprop}
\begin{align*}
&\S_{\pm} \Gamma= F, \, \,
\Gamma(0)=\Gamma_0 \notag\\
&S \phi=f, \, \,
\phi(0)=\phi_0
\end{align*}
Then, for all $\delta>0$,
\begin{align}
 \dot N_T(\Gamma)
&\lesssim_{\delta} \|<\nabla_x>^{\alpha} <\nabla_y>^{\alpha} \Gamma_0\|_{L^2(dx dy)} \label{gammaest}
+ \|<\nabla_x>^{\alpha} <\nabla_y>^{\alpha} c(s)F\|_{X_{\pm}^{-\frac{1}{2}+\delta}}\notag\\
\end{align}
\begin{align*}
N_T(\phi)
&\lesssim_{\delta} \|<\nabla_x>^{\alpha}  \phi_0\|_{L^2(dx)}
+ \|<\nabla_x>^{\alpha}c(s) f\|_{X^{-\frac{1}{2}+\delta}}
\end{align*}
\end{proposition}

The main new contribution of our current paper is  a non-obvious modification of the above estimates for $N_T(\Lambda)$, for an equation which includes the potential $\frac{1}{N}v_N(x-y)$.
The "short" version\footnote{See Theorem \ref{mnthmlong} for the "long" version of the theorem} of our main linear theorem is

\begin{theorem}\label{mnthm}
Let $0<\beta<1$, and let
\begin{align}
&\S \Lambda= \frac{1}{N}v_N(x-y) \Lambda +F \label{diffeq}\\
&\Lambda(0)=\Lambda_0 \notag
\end{align}
Then for all $\delta >0$ sufficiently small, the following holds, uniformly in $N$.
\begin{align*}
&N_T(\Lambda)\lesssim_{\delta} \|<\nabla_x>^{\alpha} <\nabla_y>^{\alpha} \Lambda_0\|
+ \|<\nabla_x>^{\alpha} <\nabla_y>^{\alpha} c(t) F\|_{X^{-\frac{1}{2}+\delta}}\\
&+ \min \{ \|<\nabla_x>^{\alpha} <\nabla_y>^{\alpha - \frac{1}{2}} c(t) F\|_{X^{-\frac{1}{4}-\delta}}, \\
 &\|<\nabla_x>^{\alpha - \frac{1}{2}} <\nabla_y>^{\alpha} c(t) F\|_{X^{-\frac{1}{4}-\delta}} \} \
\end{align*}

\end{theorem}
\begin{remark} Note that the definition of $N_T(\Lambda)$ involves a quarter time derivative, but the right hand side above involves no time derivatives.
 This time
  derivative 
  will be crucial in the proof of Theorem \ref{mnthm}, as well as future applications of the theorem to Fock space estimates.
\end{remark}

 In preparation for the proof of Theorem \ref{mainnonlinthm}, we recall some standard inequalities. For completeness, a general statement and proof are included in Lemma \eqref{s<12}.

 We will freely use the following estimates:
\begin{align}
&\|F\|_{L^{2+}(dt)L^{6-}(dx) L^2(dy)}\lesssim \|F\|_{X^{\frac{1}{2}-}}\label{Main-}\\
&\mbox{i. e.  there exist numbers $\epsilon_1, \epsilon_2, \delta>0$,
 arbitralily small, such that}\notag\\
&\|F\|_{L^{2+\epsilon_1}(dt)L^{6-\epsilon_2}(dx) L^2(dy)}\lesssim_{\epsilon_1, \epsilon_2, \delta} \|F\|_{X^{\frac{1}{2}-\delta}}\notag\\
&\|F\|_{X^{-\frac{1}{2}+}} \lesssim \|F\|_{L^{2-}(dt)L^{\frac{6}{5}+}(dx) L^2(dy)}\label{dual}\\
&\|F\|_{L^p(dt)  L^2(dx dy)} \lesssim_s \|F\|_{X^{\frac{1}{2}-}} \, \, \mbox{for arbitrarily large $p$}\notag\\
&\|F\|_{X^{-\frac{1}{2}+}} \lesssim \|F\|_{L^{1+}(dt) L^2(dx  dy)} \label{origTataru}\\
&\|F\|_{X^{-\frac{1}{4}-}} \lesssim \|F\|_{L^{\frac{4}{3}-}(dt) L^2(dx  dy)} \label{midTataru}
\end{align}

Now we are ready to prove Theorem \ref{mainnonlinthm}.

\begin{proof} The proof of Theorem \ref{mainnonlinthm} using Proposition \ref{oldprop} and assuming Theorem \ref{mnthm} (which will be proved later in the paper) is straightforward.
First, we use \ref{gammaest} with $F=$ RHS of \, \ref{evGamma2}.
We show how to prove
 \begin{align*}
& \|<\nabla_{x_1}>^{\alpha}<\nabla_{x_2}>^{\alpha} \, c(t)(\mbox{RHS of} \, \ref{evGamma2})\|_{X_{\pm}^{-\frac{1}{2}+\delta}}\\
&\lesssim  T^{\epsilon} \bigg( N^2_T\left(\Lambda\right)+
\dot N^2_T\left(\Gamma\right)+ N^4_T\left(\phi\right)\bigg)
 \end{align*}
Consider  a typical term, such as the first term (where we set $x_1-y=z$):
\begin{align*}
&  \left\{\frac{1}{i}\partial_{t}-\Delta_{x_1}+\Delta_{x_2}\right\}\overline \Gamma(x_1, x_2) \\
&=-\int dz\left\{v_N(z)\Lambda(x_1,x_1-z)\overline\Lambda(x_1-z,x_{2})\right\} + \cdots
\end{align*}
Applying the fractional Leibniz rule in $X^{\delta}$ spaces, (which follows from a trivial pointwise estimate),
\begin{align*}
&\bigg\|<\nabla_{x_1}>^{\alpha}<\nabla_{x_2}>^{\alpha}c(t)\int dz\left\{v_N(z)\Lambda(x_1,x_1-z)\overline\Lambda(x_1-z,x_{2})\right\}\bigg\|_{X_{\pm}^{-\frac{1}{2}+\delta}}\\
&\lesssim
\int dz|v_N(z)|\bigg\|c(t)\left(<\nabla_{x_1}>^{\alpha}\Lambda(x_1,x_1-z)\right)<\nabla_{x_2}>^{\alpha}
\overline\Lambda(x_1-z,x_{2})\bigg\|_{X_{\pm}^{-\frac{1}{2}+\delta}}\\
&+\int dz|v_N(z)|\bigg\|c(t)\Lambda(x_1,x_1-z)<\nabla_{x_1}>^{\alpha}<\nabla_{x_2}>^{\alpha}\overline\Lambda(x_1-z,x_{2})
\bigg\|_{X_{\pm}^{-\frac{1}{2}+\delta}}\\
&\lesssim
 T^{\epsilon}\int dz|v_N(z)|\bigg\|\left(<\nabla_{x_1}>^{\alpha}\Lambda(x_1,x_1-z)\right)<\nabla_{x_2}>^{\alpha}
c(t) \overline\Lambda(x_1-z,x_{2})\bigg\|_{L^{1+\epsilon_1}(dt)L^2(dx_1 dx_2)}\\
&+ T^{\epsilon}\int dz|v_N(z)|\bigg\|\Lambda(x_1,x_1-z)<\nabla_{x_1}>^{\alpha}<\nabla_{x_2}>^{\alpha}c(t)\overline\Lambda(x_1-z,x_{2})
\bigg\|_{L^{1+\epsilon_1}(dt)L^2(dx_1 dx_2)}\\
&\mbox{(we used H\"older in time, and formula \ref{origTataru}}\\
&\lesssim T^{\epsilon}
\left(\int|v_N(z)|\right) \sup_z\bigg\|<\nabla_{x_1}>^{\alpha}c(t)\Lambda(x_1,x_1-z)\bigg\|_{L^2(dtdx_1))}\\
&\bigg\|<\nabla_{x_2}>^{\alpha}
c(t)\overline\Lambda(x_1-z,x_{2})\bigg\|_{L^{2+\epsilon_2}(dt)L^{\infty}(dx_1)L^2( dx_2)}\\
&+ T^{\epsilon}\left(\int |v_N(z)|\right)\sup_z\bigg\|c(t)\Lambda(x_1,x_1-z)\bigg\|_{L^2(dt)L^{3+\epsilon_3}(dx_1)} \\ &\bigg\|<\nabla_{x_1}>^{\alpha}<\nabla_{x_2}>^{\alpha}c(t)\overline\Lambda(x_1-z,x_{2})
\bigg\|_{L^{2+\epsilon_2}(dt)L^{6-\epsilon_4}(dx_1)L^2( dx_2)}\\
& \lesssim T^{\epsilon} N_T(\Lambda)^2 \, \, \mbox{(we used Sobolev)}
\end{align*}
provided $\delta$ and all epsilons are sufficiently close to $0$. We do not have to keep track how the various small numbers $\delta, \epsilon_i$ are related, provided they are all much smaller that $\alpha-\frac{1}{2}$, so that the Sobolev embeddings
 $W^{\alpha, \,  6-\epsilon_4} \subset L^{\infty}$ and
 $H^{\alpha} \subset L^{3+\epsilon_3}$ are valid in $\mathbb R^3$.

Estimates for
 \begin{align*}
& \|<\nabla_{x_1}>^{\alpha}<\nabla_{x_2}>^{\alpha} \, (\mbox{RHS of} \, \ref{evLambda2})\|_{X^{-\frac{1}{2}+\delta}}
 \mbox{
 and}\\
&\|<\nabla_{x}>^{\alpha} \, (\mbox{RHS of} \, \ref{evphi})\|_{X^{-\frac{1}{2}+\delta}}
 \end{align*}
 (which are needed for $N_T(\Lambda)$ and $N_T(\phi)$) are similar.

In order to complete the proof for $N_T(\Lambda)$ we also need
 to estimate the  $X^{-\frac{1}{4}-\delta}$ norm of
$<\nabla_{x_1}>^{\alpha-\frac{1}{2}}<\nabla_{x_2}>^{\alpha}$ applied to those terms on the right hand side of \ref{evLambda2} involving $v_N(x_1-y)$. A similar argument holds for
 $X^{-\frac{1}{4}-\delta}$ norm of  $<\nabla_{x_2}>^{\alpha-\frac{1}{2}}<\nabla_{x_1}>^{\alpha}$ applied to those terms on the right hand side of \ref{evLambda2} involving $v_N(x_2-y)$.
 Recall
 \begin{align*}
 &   \left\{\frac{1}{i}\partial_{t}-\Delta_{x_1}-\Delta_{x_2}+\frac{1}{N}v_N(x_{1}-x_{2})\right\}\Lambda(x_1, x_2) \\
& =-\int dy \left\{v_N(x_{1}-y)\Gamma(y,y) + v_N(x_{2}-y)\Gamma(y,y)\right\}\Lambda(x_1, x_2) + \cdots
\end{align*}
We estimate (with suitably chosen $\delta$, $\epsilon >0$),
\begin{align*}
&\bigg\|<\nabla_{x_1}>^{\alpha-\frac{1}{2}}<\nabla_{x_2}>^{\alpha}\left(\int  v_N(x_{1}-y)\Gamma(y,y) dy\right)c(t)\Lambda(x_1, x_2)
\bigg\|_{X^{-\frac{1}{4}-\delta}}\\
&\lesssim \bigg\|\left(\int v_N(x_{1}-y)<\nabla_{y}>^{\alpha-\frac{1}{2}}\Gamma(y,y) dy \right)c(t)<\nabla_{x_2}>^{\alpha}\Lambda(x_1, x_2)
\bigg\|_{X^{-\frac{1}{4}-\delta}}\\
&+\bigg\|\int v_N(x_{1}-y)\Gamma(y,y) dy c(t)<\nabla_{x_1}>^{\alpha-\frac{1}{2}}<\nabla_{x_2}>^{\alpha}\Lambda(x_1, x_2)
\bigg\|_{X^{-\frac{1}{4}-\delta}}\\
&\lesssim  T^{\epsilon} \bigg\|\left(\int v_N(x_{1}-y)<\nabla_{y}>^{\alpha-\frac{1}{2}}\Gamma(y,y) dy \right)c(t)<\nabla_{x_2}>^{\alpha}\Lambda(x_1, x_2)
\bigg\|_{L^{\frac{4}{3}}(dt) L^2(dxdy)}\\
&+  T^{\epsilon}\bigg\|\int v_N(x_{1}-y)\Gamma(y,y) dy c(t)<\nabla_{x_1}>^{\alpha-\frac{1}{2}}<\nabla_{x_2}>^{\alpha}\Lambda(x_1, x_2)
\bigg\|_{L^{\frac{4}{3}}(dt) L^2(dxdy)}\\
&\mbox{(we used H\"older in time, and \ref{midTataru} below)}\\
&\lesssim T^{\epsilon}\|<\nabla_{x_1}>^{\alpha-\frac{1}{2}}\Gamma(x_1, x_1)\|_{L^2(dt)L^3(dx_1)}\|<\nabla_{x_2}>^{\alpha}\Lambda(x_1, x_2)\|_{L^{4}([0, T])L^{6}(dx_1)L^2 (dx_2)} \\
&+T^{\epsilon}\|\Gamma(x_1, x_1)\|_{L^2(dt)L^{3}(dx_1)}\|<\nabla_{x_1}>^{\alpha-\frac{1}{2}}<\nabla_{x_2}>^{\alpha}\Lambda(x_1, x_2)\|_{L^4([0, T])L^{6}(dx_1)L^2 (dx_2)}\\
&\lesssim T^{\epsilon} N(\Lambda)N(\Gamma)
\end{align*}
We have used Sobolev estimates such as
\begin{align*}
  &  \|<\nabla_{x_1}>^{\alpha-\frac{1}{2}}<\nabla_{x_2}>^{\alpha}\Lambda(x_1, x_2)\|_{
    L^4([0, T])L^6(dx_1)L^2(dx_2)}\\
  &\lesssim
  \|<\nabla_{x_1}>^{\alpha}<\nabla_{x_2}>^{\alpha}\Lambda(x_1, x_2)\|_{L^4([0, T])L^3(dx_1)L^2(dx_2)}
 \lesssim N_T(\Lambda) \end{align*}

Estimates for all other terms are similar.

\end{proof}

\section{Standard estimates}

We summarize some basic estimates. The following is standard, it was used (and also proved) in \cite{GM3}.
It was inspired by the estimates in \cite{C-H1}, and is
true in both $X^{\delta}$ spaces, and $X_{\pm}^{\delta}$ spaces.
\begin{lemma}\label{3.1}
Let ${\delta}> \frac{1}{2}$ and $p, q \ge 2$ be Strichartz admissible ($\frac{2}{p}+\frac{3}{q}=\frac{3}{2}$). Then
\begin{align}
 &\|F\|_{L^p(dt) L^q(dx) L^2 (d y)} \lesssim  \|F\|_{X^{\delta}} \label{XXuwenest}\\
 &\mbox{and the dual estimate is}\\
 & \|F\|_{X^{-\delta}}  \lesssim \|F\|_{L^{p'}(dt) L^{q'}(dx) L^2 (d y)} \label{XXuwenestdual}
 \end{align}
 \end{lemma}
 We will write such estimates as
 \begin{align}
 &\|F\|_{L^p(dt) L^q(dx) L^2 (d y)} \lesssim  \|F\|_{X^{\frac{1}{2}+}} \label{XXuwenest}\\
 &\mbox{and the dual estimate is}\\
 & \|F\|_{X^{-\frac{1}{2}-}}  \lesssim \|F\|_{L^{p'}(dt) L^{q'}(dx) L^2 (d y)} \label{XXuwenestdual}
 \end{align}
These estimates also hold in $x+y, x-y$ coordinates for $X^{\delta}$ spaces, but not $X_{\pm}^{\delta}$ spaces.

The following is a slightly sharper version of the estimates used in \cite{GM3}. Sharp estimates for ${\delta}<\frac{1}{2}$ have been known for a long time,
going back at least as far as
Tataru's paper \cite{Tataru}.
\begin{lemma}\label{s<12}
Let $0<{\delta}<\frac{1}{2}$ and $p, q >2$ such that
$\frac{2}{p}+\frac{3}{q} =\frac{5-4{\delta}}{2}$ Then
\begin{align*}
&\|F\|_{L^p(dt) L^q(dx) L^2(dy)} \lesssim_{\delta} \|F\|_{X^{\delta}}\\
&\|F\|_{X^{-{\delta}}} \lesssim \|F\|_{L^{p'}(dt) L^{q'}(dx) L^2(dy)}
\end{align*}
\end{lemma}

\begin{remark}
  These estimates are motivated by interpolating, formally, between the false end-point ${\delta}=\frac{1}{2}$ with $p, q$ Strichartz admissible and the trivial case ${\delta}=0$, $p=q=2$. We don't know if the above estimate corresponding to $p=2$, $\delta < \frac{1}{2}$ is correct. This is not needed for our paper, but would remove some epsilons from the exposition.

Notice that if ${\delta}< \frac{1}{2}$ is sufficiently close to $\frac{1}{2}$, $p$ and $q$ can be made arbitrarily close to Strichartz admissible pairs.

\end{remark}

For the sake of completeness, we include a proof of Lemma \ref{s<12}.

\begin{proof}
It is easier to prove the (stronger) homogeneous version, corresponding to the weight $\big|\tau+|\xi|^2 + |\eta|^2\big|^{\delta}$. So we prove that the  operator $T(F)$ defined by
\begin{align*}
\F \left(T(F)\right)(\tau, \xi, \eta) := \frac{1}{\big|\tau+|\xi|^2 + |\eta|^2\big|^{\delta}} \F (F)(\tau, \xi, \eta)
\end{align*}
maps $L^2(dt dx dy)$ to $L^p(dt) L^q(dx)L^2(dy)$. ($\F$ denotes the space-time Fourier transform). In physical space variables, this is given by
$T(F)(t, x, y)= \int k(t-s) e^{i(t-s)\left(\Delta_x+\Delta_y\right)}F(s, \cdot)ds$ where $k$ is the inverse Fourier transform of $\frac{1}{|\tau|^{\delta}}$.
By the $T T^*$ method, it suffices to show that the kernel\\
 $(k*k)(t-t') e^{i (t-t')\left(\Delta_x+\Delta_y\right)}$ maps
$L^{p'}(dt) L^{q'}(dx)L^2(dx)$ to $L^{p}(dt) L^{q}(dx)L^2(dy)$. But this follows from the known fixed time mapping properties of
$e^{i (t-t')\Delta_x}$, namely
\begin{align*}
\|e^{i (t-t')\Delta} f(x)\|_{L^q} \lesssim \frac{1}{|t-t'|^{3(\frac{1}{q'}-\frac{1}{2})}} \|f\|_{L^{q'}(\mathbb R^3)}
\end{align*}
(see \cite{taobook}),
the pointwise estimate
 $|k*k(t-t')| \lesssim \frac{1}{|t-t'|^{1-2{\delta}}}$, and Hardy-Littlewood-Sobolev. Explicitly,

\begin{align*}
&\bigg\|\big\|\|\int
(k*k)(t-t') e^{i (t-t')\left(\Delta_x+\Delta_y\right)}F(t') dt'\|_{L^2(dy)}\big\|_{L^q(dx)}\bigg\|_{L^p(dt)}\\
&
\le \bigg\|\int \big\| \|
(k*k)(t-t') e^{i (t-t')\Delta_x}F(t') \|_{L^2(dy)}\big\|_{L^q(dy)} d t'\bigg\|_{L^p(dt)}\\
&
\le \bigg\|\int \big\| \|
(k*k)(t-t') e^{i (t-t')\Delta_x}F(t') \|_{L^q(dx)}\big\|_{L^2(dy)} d t'\bigg\|_{L^p(dt)}\\
&\lesssim \bigg\|\int
\frac{1}{|t-t'|^{1-2{\delta} + 3\left(\frac{1}{q'}-\frac{1}{2}\right)}}\big\|\|F(t')\|_{L^{q'}(dx)}\big\|_{L^2(dy)} d t'
\bigg\|_{L^p(dt)}\\
&\lesssim \bigg\|\big\|\|F\|_{L^{q'}(dx)}\big\|_{L^2(dy)}\bigg\|_{L^{p'}(dt)}
\lesssim \bigg\|\big\|\|F\|_{L^2(dy)}\big\|_{L^{q'}(dx)}\bigg\|_{L^{p'}(dt)}
\end{align*}
We used Hardy-Littlewood-Sobolev and the identity
\begin{align*}
1-2{\delta} + 3\left(\frac{1}{q'}-\frac{1}{2}\right) + \frac{1}{p'}-1=\frac{1}{p}
\end{align*}

\end{proof}
We summarize the basic space-time  collapsing  estimates used in \cite{GM3}. These are inspired by estimates in
\cite{KM0, K-MMM}.
\begin{lemma} \label{spacetime}
If $\S \Lambda=0$ then
\begin{align}
&\sup_z \||\nabla|_x^{1/2} \Lambda(t, x, x+z) \|_{L^2(dt dx )} \lesssim \||\nabla|_x^{1/2} |\nabla|_y^{1/2} \Lambda_0(x, y)\|_{L^2(dx dy)}
\label{higherlambda}
\end{align}
Also, if   $\alpha>\frac{1}{2}$, $\delta>0$ then
\begin{align}
&\sup_z\|<\nabla_x>^{\alpha}\Lambda(t, x, x+z) \|_{L^2(dt dx )} \label{higherlambdaX} \\
&\lesssim \|<\nabla_x>^{\alpha} <\nabla_y>^{\alpha }\Lambda \|_{X^{1/2+ \delta}} \notag \\
\end{align}
 \end{lemma}

\section{New estimates}

\begin{proposition}\label{quartertime}
Let $\alpha>\frac{1}{2}$, and let $\Lambda(t, x, y)$, $F(t, x, y)$ such that
\begin{align*}
\Lambda= e^{it\Delta} \Lambda_0+\int_{-\infty}^{ \infty} c(t-s)e^{i(t-s)\Delta}F(s, \cdot)ds
\end{align*}
Then there exists $\epsilon>0$ such that
\begin{align*}
&\sup_z \|\big|\partial_t\big|^{\frac{1}{4}}\left(\Lambda(t, x, x+z)\right)\|_{L^2}\\
& \lesssim
  \|<\nabla_x>^{\alpha}<\nabla_y>^{\alpha}
  \Lambda_0\|_{L^2}\\
&+ \|<\nabla_x>^{\alpha}<\nabla_y>^{\alpha}
 F\|_{X^{-\frac{1+\epsilon}{2}}}
 &+\| <\nabla_x>^{\alpha}<\nabla_y>^{\alpha-\frac{1}{2}}
 F\|_{X^{-\frac{1+2 \epsilon}{4}}}
\end{align*}
\end{proposition}
\begin{proof} The proof for the solution to the homogeneous equation is similar to the proof of lemma \ref{spacetime}. The proof
for the inhomogeneous part is in the Appendix.
\end{proof}

We systematically solve
\begin{align*}
&\S \Lambda= F \\
&\Lambda(0)=\Lambda_0 \notag
\end{align*}
in an interval$[0, T]$
by constructing
\begin{align*}
&  \Lambda_1=
  e^{- i t \Delta}\Lambda_0 \\
  & +
  \int_{-\infty}^{\infty}c(t-s)e^{- i (t-s) \Delta}c(s) F(s) ds
\end{align*}
where $c(t)$ is the characteristic function of $[0, T]$, $T \le 1$,
and noticing $c(t)\Lambda=c(t)\Lambda_1$ .

We will  use
\begin{lemma}\label{energy}
Let $b>0$ and
\begin{align*}
\Lambda=\int_{-\infty}^{\infty}c(t-s)e^{- i (t-s)\Delta_{x, y}}F ds
\end{align*}
Then
\begin{align*}
&\|\Lambda\|_{X^{b}} \lesssim   \|F\|_{X^{b-1}}
\end{align*}
\end{lemma}
This is a version of the standard "$X^{s, b}$ energy estimate" (Theorem 2.12) in \cite{taobook}, which is usually stated for $b>1/2$
and uses smooth cut-off functions in time.
\begin{proof} Let $\mathcal F$ denote the space-time Fourier transform. Then

\begin{align*}
&\mathcal F \left(\int_{-\infty}^{\infty}c(t-s)e^{- i (t-s) \Delta_{x, y}}F(s) ds\right)
=\hat c (\tau + |\xi|^2 + |\eta|^2) \tilde F(\tau, \xi, \eta)\\
&\mbox{and}\\
&|\hat c (\tau + |\xi|^2 + |\eta|^2)|\lesssim \frac{1}{<\tau + |\xi|^2 + |\eta|^2>}\\
\end{align*}
\end{proof}

\begin{remark}
We will have to work in $X^b$ for both
$b>\frac{1}{2}$
and $0<b<\frac{1}{2}$, as well as their dual spaces $X^{-b}$.

One reason is that
multiplication by $c(t)$ is not bounded on $X^{\frac{1}{2}+}$, but it is bounded on  $X^{\frac{1}{2}-}$.
This is because the Fourier transform of $\chi_{[0, \infty]}$ is a singular integral operator,   $|\tau|^{b} $ is in the class $A_2$ iff $-1<b<1$. See Stein's book \cite{Stein},  section 4 (Chapter 5) and remark 6.4 on page 218 of \cite{Stein}.

We want, for $0< b<1$
\begin{align*}
\|c(t) F\|_{X^b}=
\| \hat c* \tilde F\|_{L^2(<\tau+|\xi|^2+|\eta|^2>^{b})d \tau d \xi d\eta}\lesssim \| \tilde F\|_{L^2(<\tau+|\xi|^2+|\eta|^2>^{b}d \tau d \xi d \eta)}
\end{align*}
It suffices to show
\begin{align*}
\| \hat c* \tilde F\|_{L^2(|\tau+|\xi|^2+|\eta|^2|^{b})d \tau d \xi d\eta}\lesssim \| \tilde F\|_{L^2(|\tau+|\xi|^2+|\eta|^2|^{b}d \tau d \xi d \eta)}
\end{align*}
The convolution is in just one dimension, so, after changing variables, this follows from
\begin{align*}
\|\hat c* f\|_{L^2(|\tau|^{b}d \tau )}\lesssim \|  f\|_{L^2(|\tau|^{b}d \tau )}
\end{align*}
for $f=f(t)$.
This weighted $L^2$ estimate is one of the properties of $A_2$ weights. We remark that by the same argument
\begin{align*}
\||\partial_t|^{\frac{b}{2}}\left(c(t) f\right)\|_{L^2(dt dx)} \lesssim \||\partial_t|^{\frac{b}{2}} f\|_{L^2(dt dx)}
\end{align*}
for $0<b<1$.
\end{remark}

Next, we need to define projections such as $P_{|\xi-\eta|\lesssim N^{\beta'}}$. These have to be bounded on $L^p$ spaces, so they have to involve smooth cut-offs.

Let $\varphi\in C_0^{\infty}(-2, 2)$ be identically $1$ on $[-1, 1]$. Also, assume $2^I < N^{\beta'} \le 2^{I+1}$.

We define the operators
\begin{align*}
  &\left(P_{|\xi-\eta| \lesssim N^{\beta'}} F \right)(x, y) \\
  &=\left(P_{|\xi-\eta| < 2^I} F \right)(x, y) = \F^{-1} \left(
  \varphi\left(\frac{|\xi-\eta|}{2^I}\right)\F (F)(\xi, \eta)\right)\\
&\left(P_{|\xi-\eta| \gtrsim N^{\beta'}} F \right)(x, y) = F(x, y)- \left(P_{|\xi-\eta| \lesssim N^{\beta'}} F \right)(x, y)
\end{align*}
where $\F$ denotes the Fourier transform. Similarly, we define $P_{|\xi+\eta| \lesssim N^{\beta'}} F$ and
$P_{|\xi+\eta| \gtrsim N^{\beta'}} F$. These operators are bounded on all (mixed) $L^p(dx) L^q(dy)$ spaces.

The proof of Theorem \ref{mnthm} will be based on more technical estimates, using a more complicated norm:
(denoted by $\N=\N_T$, as opposed to the more elementary norms $N_T$ from the statement of Theorem \ref{mnthm}).
Recall $0 \le \beta <1$ and $\alpha >\frac{1}{2}$  is such that $2 \alpha \beta <1$.
Let $\frac{1}{2}-$ be a number such that  $\frac{1}{2}-<\frac{1}{2}$ and $\frac{1}{2}- \left(\frac{1}{2}-\right) << \alpha-\frac{1}{2}<<1-\beta$.

Define the norm

\begin{align*}
&\N(\Lambda)=
\|<\nabla_x>^{\alpha}<\nabla_y>^{\alpha}P_{|\xi-\eta| \gtrsim N^{\beta'}}c(t)\Lambda\|_{X^{\frac{1}{2}-}}\\
&+\|<\nabla_x>^{\alpha}<\nabla_y>^{\alpha}P_{|\xi+\eta| \gtrsim N^{\beta'}}c(t)\Lambda\|_{X^{\frac{1}{2}-}}\\
&+\|<\nabla_x>^{\alpha}<\nabla_y>^{\alpha}P_{|\xi-\eta|\lesssim N^{\beta'}}P_{|\xi+\eta|\lesssim N^{\beta'}}c(t)\Lambda\|_{L^{2}(dt) L^{6}( dx) L^2(dy)}\\
&+\|<\nabla_x>^{\alpha}<\nabla_y>^{\alpha}P_{|\xi-\eta|\lesssim N^{\beta'}}P_{|\xi+\eta|\lesssim N^{\beta'}}c(t)\Lambda\|_{L^{\infty}(dt) L^{2}( dx) L^2(dy)}\\
&+\mbox{same norm with $x$ and $y$ reversed}\\
&+\sup_w\|<\nabla>^{\alpha} c(t)\Lambda (t, x+w, x-w)\|_{L^2(dt dx)}\\
&+\sup_w\|\big|\partial_t\big|^{1/4} c(t)\Lambda (t, x+w, x-w)\|_{L^2(dt dx)} \\
&+\|c(t)\Lambda\|_{X^{\frac{1}{2}-}}+N^{-1}\|<\nabla_x>^{\alpha}<\nabla_y>^{\alpha}P_{|\xi-\eta|\lesssim N^{\beta'}, |\xi+\eta|\lesssim N^{\beta'}}c(t)\Lambda\|_{X^{\frac{1}{2}-}}\\
\end{align*}

The power $N^{-1}$ in the last term could be replaced by any $N^{-k}$, and it is meant to handle error terms.

The "long" version of our main linear theorem is

\begin{theorem}\label{mnthmlong}

Let $c$ be the characteristic function of $[0, T]$, and let $\Lambda$ satisfy the integral equation
\begin{align}
&  \Lambda=
  c(t)e^{- i t \Delta}\Lambda_0 \notag\\
  & +
  \frac{1}{N} \int_{-\infty}^{\infty}c(t-s)e^{- i (t-s) \Delta}v_N c(s)\Lambda(s) ds+
  \int_{-\infty}^{\infty}c(t-s)e^{- i (t-s) \Delta}c(s) F(s) ds \notag\\
  &=A+B+C \label{defB}
\end{align}
so that $\Lambda$ agrees with the solution of \ref{diffeq} in $[0, T]$. Then
\begin{align*}
&\N(\Lambda)\lesssim
\|<\nabla_x>^{\alpha} <\nabla_y>^{\alpha} \Lambda_0\|_{L^2}
+ \|<\nabla_x>^{\alpha} <\nabla_y>^{\alpha} c(t) F\|_{X^{-\frac{1}{2}+}}\\
&+ \min \{ \|<\nabla_x>^{\alpha} <\nabla_y>^{\alpha-\frac{1}{2}} F\|_{X^{-\frac{1}{4}-}},
 \|<\nabla_x>^{\alpha-\frac{1}{2}}<\nabla_y>^{\alpha} F\|_{X^{-\frac{1}{4}-}} \}
\end{align*}
\end{theorem}

\begin{remark} Before proving the above theorem, we point out that it implies Theorem \ref{mnthm}.
The right hand sides of the two inequalities are the same.

As for the left hand sides, if $\Lambda$ is as in the statement of the above theorem, then
\begin{align*}
  N_T(\Lambda) \lesssim \N(\Lambda) +
\|<\nabla_x>^{\alpha} <\nabla_y>^{\alpha} \Lambda_0\|_{L^2}
+ \|<\nabla_x>^{\alpha} <\nabla_y>^{\alpha} c(t) F\|_{X^{-\frac{1}{2}+}}
\end{align*}
Indeed, recall
\begin{align*}
 N_T(\Lambda)=&\|<\nabla_x>^{\alpha}<\nabla_y>^{\alpha} c(t)\Lambda\|_{L^{2}(dt) L^{6}( dx) L^2(dy)}\\
 &+\|<\nabla_x>^{\alpha}<\nabla_y>^{\alpha}c(t)\Lambda\|_{L^{4}(dt) L^{3}( dx) L^2(dy)}\\
&+\mbox{same norm with $x$ and $y$ reversed}\\
&+\sup_w\|<\nabla>^{\alpha} c(t)\Lambda (t, x+w, x-w)\|_{L^2(dt dx)}
\end{align*}
We have
\begin{align*}
  &N(A) + N(C) \lesssim
  \|<\nabla_x>^{\alpha} <\nabla_y>^{\alpha} \Lambda_0\|_{L^2}
+ \|<\nabla_x>^{\alpha} <\nabla_y>^{\alpha} c(t) F\|_{X^{-\frac{1}{2}+}}
\end{align*}
(see lemmas \ref{3.1}, \ref{spacetime}, \ref{energy}).
 Also, using the definition of $B$ \ref{defB}, we have a small improvement in $B$ over $\Lambda$ at high frequencies:
 \begin{align*}
& \|<\nabla_x>^{\alpha}<\nabla_y>^{\alpha}P_{|\xi\pm\eta| \gtrsim N^{\beta'}}B\|_{X^{\frac{1}{2}+}}\\
& \lesssim N^{-\epsilon}\|<\nabla_x>^{\alpha}<\nabla_y>^{\alpha}P_{|\xi\pm\eta| \gtrsim N^{\beta'}}c(t)\Lambda\|_{X^{\frac{1}{2}-}}
 \end{align*}

 Indeed, using Lemma \ref{energy} and \ref{dual},
 \begin{align*}
& \|<\nabla_x>^{\alpha}<\nabla_y>^{\alpha}P_{|\xi-\eta| \gtrsim N^{\beta'}}B\|_{X^{\frac{1}{2}+}}
 \lesssim
\frac{1}{N}  \|<\nabla_x>^{\alpha}<\nabla_y>^{\alpha}P_{|\xi-\eta| \gtrsim N^{\beta'}}v_N c\Lambda\|_{X^{-\frac{1}{2}+}}\\
&\lesssim \frac{1}{N}  \|<\nabla_x>^{\alpha}<\nabla_y>^{\alpha}P_{|\xi-\eta| \gtrsim N^{\beta'}}v_N c\Lambda\|_{L^{2-}(dt)L^{\frac{6}{5}+}(d(x-y))L^2(d(x+y))}\\
&\lesssim \|<\nabla_x>^{\alpha}<\nabla_y>^{\alpha}P_{|\xi-\eta| \gtrsim N^{\beta'}}c(t)\Lambda\|_{X^{\frac{1}{2}-}}
\end{align*}
We used the fact that, because of the frequency localization $|\xi-\eta| \gtrsim N^{\beta'}$, at least one of
 $|\xi|$, $|\eta|\gtrsim N^{\beta'}$, while  $v_N$ is supported, in Fourier space, at frequencies $\le N^{\beta}$.
 Thus at least one of
 $<\nabla_x>^{\alpha}, <\nabla_y>^{\alpha}$
falls on $\Lambda$, and then we used \ref{dual}.
For instance,
\begin{align*}
&\frac{1}{N}  \|P_{|\xi-\eta| \gtrsim N^{\beta'}}v_N <\nabla_x>^{\alpha}<\nabla_y>^{\alpha}c\Lambda\|_{X^{-\frac{1}{2}+}}\\
&\lesssim\frac{1}{N}  \|P_{|\xi-\eta| \gtrsim N^{\beta'}}v_N c<\nabla_x>^{\alpha}<\nabla_y>^{\alpha}\Lambda\|_{L^{2-}(dt)L^{\frac{6}{5}+}(d(x-y))L^2(d(x+y))}\\
&\lesssim \left(\frac{1}{N} \|v_N\|_{L^{\frac{3}{2}+}}\right)\|c<\nabla_x>^{\alpha}<\nabla_y>^{\alpha}\Lambda\|_{L^{2+}(dt)L^{6-}(d(x-y))L^2(d(x+y))}\\
&\lesssim N^{- \epsilon}\|<\nabla_x>^{\alpha}<\nabla_y>^{\alpha}P_{|\xi-\eta| \gtrsim N^{\beta'}}c(t)\Lambda\|_{X^{\frac{1}{2}-}}
\end{align*}
The term where one  $<\nabla>^{\alpha}$ falls on the potential is handled in a similar way, and so are
the terms involving $P_{|\xi+\eta| \gtrsim N^{\beta'}}$.

\end{remark}

\begin{proof}(of Theorem \ref{mnthmlong}).

The outline of the proof will be
\begin{align*}
&\N(\Lambda)\lesssim
\N(A)+ \N(C)\\
&+
\N(P_{|\xi-\eta|\gtrsim N^{\beta'}}B) + \N(P_{|\xi+\eta|\gtrsim N^{\beta'}}B) +
\N(P_{|\xi-\eta|\lesssim N^{\beta'}, |\xi+\eta|\lesssim N^{\beta'}}B)\\
&\mbox{and we will show each individual term is}\\
&\lesssim N^{-\epsilon} \N(\Lambda) + \|<\nabla_x>^{\alpha} <\nabla_y>^{\alpha} \Lambda_0\|
+ \|<\nabla_x>^{\alpha} <\nabla_y>^{\alpha}c(t) F\|_{X^{-\frac{1}{2}+}}\\
&+ \min \{ \|<\nabla_x>^{\alpha} <\nabla_y>^{\alpha-\frac{1}{2}} F\|_{X^{-\frac{1}{4}-}},
 \|<\nabla_x>^{\alpha-\frac{1}{2}} <\nabla_y>^{\alpha} F\|_{X^{-\frac{1}{4}-}} \}
\end{align*}
Using Lemma \ref{spacetime} and Proposition \ref{quartertime} the above is already known for the $A$ and $C$ parts of $\Lambda$. More precisely, we have

\begin{align*}
&\N(A)\lesssim \|<\nabla_x>^{\alpha} <\nabla_y>^{\alpha} \Lambda_0\|_{L^2}\\
&\N(C) \lesssim  \|<\nabla_x>^{\alpha} <\nabla_y>^{\alpha}c(t) F\|_{X^{-\frac{1}{2}+}}+
\|<\nabla_x>^{\alpha} <\nabla_y>^{{\alpha} -\frac{1}{2}}c(t) F\|_{X^{-\frac{1}{4}-}}
\end{align*}

 All work will be devoted to proving

 \begin{align*}
&\N(P_{|\xi-\eta|\gtrsim N^{\beta'}}B) + \N(P_{|\xi+\eta|\gtrsim N^{\beta'}}B) \\
&\lesssim N^{-\epsilon}\bigg( \N(\Lambda) + \|<\nabla_x>^{\alpha} <\nabla_y>^{\alpha} \Lambda_0\|
+ \|<\nabla_x>^{\alpha} <\nabla_y>^{\alpha} F\|_{X^{-\frac{1}{2}+}}\bigg)\\
&\N(P_{|\xi-\eta|\lesssim N^{\beta'}, |\xi+\eta|\lesssim N^{\beta'}}B)
\lesssim N^{-\epsilon} \N(\Lambda)
\end{align*}

This will be split in several sections.

Whenever possible, we will estimate $B$ (localized in frequency space and differentiated) using
Lemma \ref{energy}, followed by H\"older's inequality and the Strichartz type estimates of Lemmas \ref{3.1} and \ref{s<12}.

Let
\begin{align*}
B=\frac{1}{N}\int_{-\infty}^{\infty}c(t-s)e^{- i (t-s) \Delta}v_N c(s)\Lambda(s) ds
\end{align*}

\begin{align*}
&\|B\|_{X^{\frac{1}{2}-}} \lesssim \frac{1}{N}\|v_N c(s) \Lambda\|_{X^{-\frac{1}{2}-}}\lesssim  \frac{1}{N}\|v_N c(s) \Lambda\|_{L^2(dt)L^{\frac{6}{5}}(d(x-y))L^2(d(x+y))}\\
&\lesssim \left( \frac{1}{N}\|v_N \|_{L^{\frac{3}{2}+}}\right) \|c(s)  \Lambda\|_{L^2(dt)L^{6-}(d(x-y)) L^2(d(x+y))}\\
&\lesssim N^{-\epsilon}\|c(s)  \Lambda\|_{X^{\frac{1}{2}-}} \lesssim N^{-\epsilon}\N(\Lambda)
\end{align*}
with
$\frac{1}{N}\|v_N \|_{L^{\frac{3}{2}+}}\sim N^{-\epsilon}$.

This is representative of the proof that follows provided at least one of $<\nabla_x>^{\alpha}$, $<\nabla_y>^{\alpha}$ fall
on $\Lambda$.

Also, from the above we can read off

\begin{align*}
&\|c(t) \Lambda\|_{X^{\frac{1}{2}-}} \lesssim
\| \Lambda_0\|_{L^2}
+ \|c(t) F\|_{X^{-\frac{1}{2}+}} + N^{-\epsilon}\N(\Lambda)
\end{align*}
and therefore (since $|\xi|$, $|\eta| \lesssim N^{\beta'}$ on the Fourier space support of $P_{|\xi-\eta|\lesssim N^{\beta'}}P_{|\xi+\eta|\lesssim N^{\beta'}}c(t)\Lambda$)

\begin{align*}
  &\|<\nabla_x>^{\alpha }<\nabla_y>^{\alpha}P_{|\xi-\eta|\lesssim N^{\beta'}}P_{|\xi+\eta|\lesssim N^{\beta'}}c(t)\Lambda\|_{X^{\frac{1}{2}-}}\\
&\lesssim N^{2 \alpha \beta'} \left(
\| \Lambda_0\|_{L^2}
+ \|c(t) F\|_{X^{-\frac{1}{2}+}}  + N^{-\epsilon}\N(\Lambda)\right)
\end{align*}

On the other hand, the terms
$\|<\nabla_x>^{\alpha }<\nabla_y>^{\alpha}P_{|\xi -\eta|\gtrsim N^{\beta'}} c(t)\Lambda\|_{X^{\frac{1}{2}-}}$ and
$\|<\nabla_x>^{\alpha }<\nabla_y>^{\alpha}P_{|\xi + \eta|\gtrsim N^{\beta'}} c(t)\Lambda\|_{X^{\frac{1}{2}-}}$

are part of $\N(\Lambda)$,
so we have to estimate the last term in $\N$:
\begin{align*}
&N^{-1}\|<\nabla_x>^{\alpha}<\nabla_y>^{\alpha}c(t)\Lambda\|_{X^{\frac{1}{2}-}} \\
&\lesssim
\| \Lambda_0\|_{L^2}
+ \|c(t) F\|_{X^{-\frac{1}{2}+}}  + N^{-\epsilon}\N(\Lambda)
\end{align*}

The rest of the proof of theorem \ref{mnthmlong} will be split into several sections.

\section{Estimates for $\N(B)$ at  frequency $|\xi-\eta|\gtrsim N^{\beta'}$ or $|\xi+\eta|\gtrsim N^{\beta'}$}

Before going into details, let us point out that in this region at least one of $|\xi|, |\eta| \gtrsim N^{\beta'}$. On the other hand, $v_N(x-y)$ is supported at frequencies $\le N^{\beta}<<N^{\beta'}$. Thus, heuristically, at least,
in the expression
\begin{align*}
&<\nabla_x>^{\alpha}<\nabla_y>^{\alpha} P_{|\xi\pm\eta|\gtrsim N^{\beta'}}\left(v_N \Lambda\right)\\
&\sim P_{|\xi\pm\eta|\gtrsim N^{\beta'}}\left((<\nabla_x>^{\alpha}v_N) <\nabla_y>^{\alpha} \Lambda\right)\\
&+P_{|\xi\pm\eta|\gtrsim N^{\beta'}}\left((<\nabla_y>^{\alpha}v_N) <\nabla_x>^{\alpha} \Lambda\right)\\
&+ P_{|\xi\pm\eta|\gtrsim N^{\beta'}}\left(v_N <\nabla_x>^{\alpha}<\nabla_y>^{\alpha} \Lambda\right)
\end{align*}
These  terms can be treated using Strichartz estimates and $X^{\pm\frac{1}{2}}$ techniques.

Recall
\begin{align*}
B=
  \frac{1}{N} \int_{-\infty}^{\infty}c(t-s)e^{- i (t-s) \Delta}v_N c(s)\Lambda(s) ds
\end{align*}

Also, we have to use the more precise notation $P_{|\xi-\eta|\gtrsim N^{\beta'}}= P_{|\xi-\eta|> 2^I}$
where we recall $N^{\beta'} \sim 2^I$. Notice $N^{\beta} \sim 2^{I(\beta-\beta')}N^{\beta'}$, and $I(\beta'-\beta) \to \infty $ as $N\to \infty$. Also, $N^{\beta} << N^{\beta'}$.

In this section we  prove
\begin{proposition} The following estimates hold:
\begin{align*}
&\N\left(P_{|\xi+\eta|\gtrsim N^{\beta'}} B\right)\\
&\lesssim
\|P_{|\xi+\eta|\gtrsim N^{\beta'}} <\nabla_x>^{\alpha}<\nabla_y>^{\alpha}B\|_{X^{\frac{1}{2}+}}\\
&\lesssim N^{- \epsilon} \|P_{|\xi+\eta|\gtrsim N^{\beta'}} <\nabla_x>^{\alpha}<\nabla_y>^{\alpha}c(t)\Lambda\|_{X^{\frac{1}{2}-}}\\
&\lesssim  N^{- \epsilon} \N(\Lambda)\\
&\N\left(P_{|\xi-\eta|\gtrsim N^{\beta'}} B\right)\\
&\lesssim\|P_{|\xi-\eta|\gtrsim N^{\beta'}} <\nabla_x>^{\alpha}<\nabla_y>^{\alpha}B\|_{X^{\frac{1}{2}+}}\\
&\lesssim N^{- \epsilon}\bigg(
 \|<\nabla_x>^{\alpha} <\nabla_y>^{\alpha} \Lambda_0\|
+ \|<\nabla_x>^{\alpha} <\nabla_y>^{\alpha} c(t) F\|_{X^{-\frac{1}{2}+}}\\
&+ N^{-1} \|<\nabla_x>^{\alpha}<\nabla_y>^{\alpha}c(t)\Lambda\|_{X^{\frac{1}{2}-}}\bigg)\\
&\lesssim N^{- \epsilon}\bigg(
 \|<\nabla_x>^{\alpha} <\nabla_y>^{\alpha} \Lambda_0\|
+ \|<\nabla_x>^{\alpha} <\nabla_y>^{\alpha} c(t) F\|_{X^{-\frac{1}{2}+}}\\
&+ \N(\Lambda)\bigg)\\
\end{align*}

\end{proposition}
 \begin{proof}

We have, using the boundedness of multiplication by $c(t)$ on $X^{\frac{1}{2}-}$,
\begin{align*}
&\|P_{|\xi-\eta|>2^I}<\nabla_x>^{\alpha}<\nabla_y>^{\alpha}c(t)B\|_{X^{\frac{1}{2}-}}
\lesssim \|P_{|\xi-\eta|>2^I}<\nabla_x>^{\alpha}<\nabla_y>^{\alpha}B\|_{X^{\frac{1}{2}+}}
\end{align*}
We will use the fact that the Fourier support of the product of two functions is the algebraic sum of the two supports.
Thus $P_{|\xi-\eta|>2^I}\big(v_N c(s) \Lambda(s)\big)=
P_{|\xi-\eta|>2^I}\big(v_N c(s) P_{|\xi-\eta|>2^{I-1}}\Lambda(s)\big) $, since $v_N$ is supported in $|\xi-\eta|<<2^{I}$.
We  estimate
\begin{align}
& \|P_{|\xi-\eta|>2^I}<\nabla_x>^{\alpha}<\nabla_y>^{\alpha}B\|_{X^{\frac{1}{2}+}}\notag\\
&= \| \frac{1}{N} \int_{-\infty}^{\infty}c(t-s)e^{- i (t-s) \Delta}P_{|\xi-\eta|>2^I}<\nabla_x>^{\alpha}<\nabla_y>^{\alpha}\big(v_N c(s) P_{|\xi-\eta|>2^{I-1}}\Lambda(s) \big)ds\|_{X^{\frac{1}{2}+}}\notag\\
&\lesssim  \frac{1}{N} \|<\nabla_x>^{\alpha}<\nabla_y>^{\alpha}\big(v_N c(s) P_{|\xi-\eta|>2^{I-1}}\Lambda(s)\|_{X^{\left(-\frac{1}{2}\right)+}} \label{B}
\end{align}

In the $L^2$ based space $X^{-\frac{1}{2}+}$ we can distribute the derivatives on the Fourier transform side.  Since at least one of $|\xi|$, $|\eta|$ is $\gtrsim N^{\beta'}$, at least one of $<\nabla_x>^{\alpha}, \, <\nabla_y>^{\alpha}$ lands on $\Lambda$. Denote $ \lfloor \Lambda \rfloor = \F^{-1} |\F \Lambda|$ ($\F$ is the Fourier transform in $x$ and $y$), and recall we assumed $|\hat v| \le \hat w$
for some Schwartz function $w$. We have

\begin{align*}
 & \frac{1}{N} \|<\nabla_x>^{\alpha}<\nabla_y>^{\alpha}\big(v_N c(s) P_{|\xi-\eta|>2^{I-1}}\Lambda(s) \big)\|_{X^{-\frac{1}{2}+}}\\
 &\lesssim  \frac{1}{N}  \|\big((<\nabla_x>^{\alpha}w_N ) P_{|\xi-\eta|>2^{I-1}}<\nabla_y>^{\alpha}\lfloor c(s)\Lambda(s) \rfloor \big)\|_{X^{-\frac{1}{2}+}}\\
 &+ \frac{1}{N} \|\big((<\nabla_y>^{\alpha}w_N ) P_{|\xi-\eta|> 2^{I-1}}<\nabla_x>^{\alpha}\lfloor c(s)\Lambda(s) \rfloor \big)\|_{X^{-\frac{1}{2}+}}\\
 &+ \frac{1}{N} \|w_N c(s) P_{|\xi-\eta|>2^{I-1}}<\nabla_x>^{\alpha}<\nabla_y>^{\alpha}\lfloor c(s)\Lambda(s) \rfloor \|_{X^{-\frac{1}{2}+}}\\
 &\lesssim N^{-\epsilon} \|<\nabla_x>^{\alpha}<\nabla_y>^{\alpha} P_{|\xi-\eta|>2^{I-1}} \lfloor c(s)\Lambda(s) \rfloor\|_{L^{2+}(dt) L^{6-}( d(x-y)) L^2(d(x+y))}\\
 &\lesssim N^{-\epsilon} \|<\nabla_x>^{\alpha}<\nabla_y>^{\alpha} P_{|\xi-\eta|>2^{I-1}} \lfloor c(s)\Lambda(s) \rfloor\|_{X^{\frac{1}{2}-}}\\
 &=N^{-\epsilon} \|<\nabla_x>^{\alpha}<\nabla_y>^{\alpha} P_{|\xi-\eta|>2^{I-1}}  c(s)\Lambda(s) \|_{X^{\frac{1}{2}-}}\\
  \end{align*}
  We have used Holder's inequality and Lemmas \ref{3.1}, \ref{s<12}. For instance,
  \begin{align*}
 & \frac{1}{N} \|w_N P_{|\xi-\eta|>2^{I-1}}<\nabla_x>^{\alpha}<\nabla_y>^{\alpha}\lfloor c(s)\Lambda(s) \rfloor \|_{X^{-\frac{1}{2}+}}\\
 & \lesssim \frac{1}{N} \|w_N P_{|\xi-\eta|>2^{I-1}}<\nabla_x>^{\alpha}<\nabla_y>^{\alpha}\lfloor c(s)\Lambda(s) \rfloor \|_{L^{2-}(dt) L^{\frac{6}{5}+}(d(x-y) L^2(d(x+y)}\\
 & \lesssim \frac{1}{N} \|w_N \|_{L^{\frac{3}{2}+}}\| P_{|\xi-\eta|>2^{I-1}}<\nabla_x>^{\alpha}<\nabla_y>^{\alpha}\lfloor c(s)\Lambda(s) \rfloor \|_{L^{2+}(dt) L^{6-}(d(x-y) L^2(d(x+y)}\\
 &\le N^{-\epsilon} \|<\nabla_x>^{\alpha}<\nabla_y>^{\alpha} P_{|\xi-\eta|>2^{I-1}}  c(s)\Lambda(s) \|_{X^{\frac{1}{2}-}}
\end{align*}
where the numbers are chosen so that $\frac{1}{N} \|w_N \|_{L^{\frac{3}{2}+}}\sim  N^{-\epsilon}$.

 The last term is not part of $\N$ because of $ P_{|\xi-\eta|>2^{I-1}} $. However, we can iterate this result
 essentially $K=I(\beta'-\beta)$ times. The iteration stops when $2^{I-K}$ is comparable to $N^{\beta'}$, the frequency of $v_N$.
Then we cannot put at least one $<\nabla>^{\alpha}$ on $\Lambda$, and the argument breaks down.

Proceeding this way $K$ times,

\begin{align*}
& \|P_{|\xi-\eta|>2^I}<\nabla_x>^{\alpha}<\nabla_y>^{\alpha}B\|_{X^{\frac{1}{2}+}}\\
&\lesssim N^{-\epsilon} \|<\nabla_x>^{\alpha}<\nabla_y>^{\alpha} P_{|\xi-\eta|>2^{I-1}}  c(s)\Lambda(s) \|_{X^{\frac{1}{2}-}}\\
& \lesssim  N^{-\epsilon} \bigg(\|<\nabla_x>^{\alpha} <\nabla_y>^{\alpha} \Lambda_0\|
+ \|<\nabla_x>^{\alpha} <\nabla_y>^{\alpha} c(t) F\|_{X^{-\frac{1}{2}+}}\\
&+\|<\nabla_x>^{\alpha}<\nabla_y>^{\alpha}P_{|\xi-\eta|>2^{I-1}}B\|_{X^{\frac{1}{2}+}}\bigg)\\
&\lesssim   N^{-\epsilon} \bigg(\|<\nabla_x>^{\alpha} <\nabla_y>^{\alpha} \Lambda_0\|
+ \|<\nabla_x>^{\alpha} <\nabla_y>^{\alpha} c(t) F\|_{X^{-\frac{1}{2}+}}\\
&+N^{-\epsilon}\|<\nabla_x>^{\alpha}<\nabla_y>^{\alpha}P_{|\xi-\eta|>2^{I-2}}c(t)\Lambda\|_{X^{\frac{1}{2}-}}\big)  \bigg)\\
&\cdots\\
&\lesssim N^{-\epsilon} \bigg( \|<\nabla_x>^{\alpha} <\nabla_y>^{\alpha} \Lambda_0\|
+ \|<\nabla_x>^{\alpha} <\nabla_y>^{\alpha} c(t) F\|_{X^{-\frac{1}{2}+}}\\
&+ N^{-K \epsilon} \|<\nabla_x>^{\alpha}<\nabla_y>^{\alpha}c(t)\Lambda\|_{X^{\frac{1}{2}-}}\bigg)\\
&\lesssim N^{-\epsilon} \bigg(\|<\nabla_x>^{\alpha} <\nabla_y>^{\alpha} \Lambda_0\|
+ \|<\nabla_x>^{\alpha} <\nabla_y>^{\alpha} c(t) F\|_{X^{-\frac{1}{2}+}}\\
&+N^{-\epsilon} \N(\Lambda)\bigg)
\end{align*}
provided $K \epsilon >1$.
The argument for $ P_{|\xi+\eta|\gtrsim N^{\beta'}}$ is similar, but easier (this closes in one step, no need to iterate).
\end{proof}
Notice that we have in fact shown
\begin{corollary}
\begin{align*}
&\|<\nabla_x>^{\alpha}<\nabla_y>^{\alpha}P_{|\xi-\eta|>2^I}B\|_{X^{\frac{1}{2}+}}\\
&+\|<\nabla_x>^{\alpha}<\nabla_y>^{\alpha}P_{|\xi+\eta|>2^I}B\|_{X^{\frac{1}{2}+}}\\
&\lesssim N^{-\epsilon} \bigg(\|<\nabla_x>^{\alpha} <\nabla_y>^{\alpha} \Lambda_0\|
+ \|<\nabla_x>^{\alpha} <\nabla_y>^{\alpha} c(t) F\|_{X^{-\frac{1}{2}+}} \\
&+N^{-\epsilon} \N(\Lambda)\bigg)
\end{align*}
and thus
\begin{align*}
&\N(P_{|\xi\pm\eta|>2^I}B) \lesssim \|<\nabla_x>^{\alpha}<\nabla_y>^{\alpha}P_{|\xi\pm\eta|>2^I}B\|_{X^{\frac{1}{2}+}}\\
&\lesssim N^{-\epsilon} \bigg(\|<\nabla_x>^{\alpha} <\nabla_y>^{\alpha} \Lambda_0\|
+ \|<\nabla_x>^{\alpha} <\nabla_y>^{\alpha} c(t) F\|_{X^{-\frac{1}{2}+}}\\
&+N^{-\epsilon} \N(\Lambda)\bigg)
\end{align*}
\end{corollary}

\section{The collapse in the region $|\xi-\eta|\lesssim N^{\beta'}$ and $|\xi+\eta|\lesssim N^{\beta'}$}

Heuristically, in this region the frequencies of $\Lambda$ are less than the frequencies of the potential term, so the worst case scenario is
\begin{align*}
&<\nabla_x>^{\alpha}<\nabla_y>^{\alpha} P_{|\xi\pm\eta|\lesssim N^{\beta'}}\left(\frac{1}{N}v_N \Lambda\right)\\
&\sim P_{|\xi\pm\eta|\lesssim N^{\beta'}}\left(\left(<\nabla_x>^{\alpha}<\nabla_y>^{\alpha}\frac{1}{N}v_N\right) \Lambda\right)\\
&\sim N^{-\epsilon}\delta(x-y) \Lambda(t, \frac{x+y}{2}, \frac{x+y}{2})
\end{align*}
The proof that follows exploits this structure. It will only use the condition$|\xi-\eta|\lesssim N^{\beta'}$.

We will prove

\begin{align*}
  &\sup_z\|<\nabla_x>^{\alpha}(P_{|\xi-\eta|\lesssim N^{\beta'}} B)
  (t, x, x+z)\|_{L^2(dt dx)}\\
   &+ \sup_z\|\bigg|\partial_t\bigg|^{\frac{1}{4}}(P_{|\xi-\eta|\lesssim N^{\beta'}} B)
  (t, x, x+z)\|_{L^2(dt dx)} \lesssim N^{-\epsilon} \N(\Lambda)
\end{align*}

Before estimating the general case
\begin{align*}
B=
  \frac{1}{N} \int_{-\infty}^{\infty}c(t-s)e^{- i (t-s) \Delta}v_N c(s)\Lambda(s) ds
\end{align*}
we consider the special case where
$ v_N \Lambda$ is replaced by $F(t, x, y)=F_z(t, x, y)=\delta(x-y-z)f(t, x+y)$ ($z$ fixed). Then we will use
$f(t, x+y)=f_z(t, x+y)=\Lambda(t, \frac{x+y+z}{2}, \frac{x+y-z}{2})$.
In fact, we suggest the reader takes $v_N \Lambda=\delta(x-y)\Lambda\left(\frac{x+y}{2}, \frac{x+y}{2}\right)$, which, while not literally true, is representative of the rigorous proof.

\begin{lemma} Fix $z$, $w$, let $M=N^{\beta'}$,  and let
\begin{align*}
B_z=
  \frac{1}{N} \int_{-\infty}^{\infty}c(t-s)e^{- i (t-s) \Delta}\delta(x-y-z)c(s)f(s, x+y) ds
\end{align*}

  Then (denoting by $\F(F)$ the Fourier transform in $t$ and $x$, and by $\tilde F$ the Fourier transform in $t$, $x$ and $y$), we have the pointwise estimate, uniformly in $z$ and $w$:
  \begin{align*}
 \bigg| \F\left(P_{|\xi-\eta|< M}( B_z) (t, x+w, x-w)\right)\bigg| (\tau, \xi)
      \lesssim  \frac{ M \log M}{N} | \F (cf)|(\tau, \xi)
      \end{align*}
      which trivially imply the $L^2$ estimates
 \begin{align*}
& \| <\nabla_x>^{\alpha} \bigg(P_{|\xi-\eta|< M}\left(c(t)B_z\right)\bigg) (t, x+w, x-w)\|_{L^2(dt dx)}
      \lesssim \frac{ M \log M}{N} \|<\nabla_x>^{\alpha}(cf)\|_{L^2(dt dx)}\\
      &\|\bigg|\partial_t\bigg|^{\frac{1}{4}}(P_{|\xi-\eta|<M} B)
  (t, x, x+z)\|_{L^2(dt dx)}\lesssim \frac{ M \log M}{N} \|\bigg|\partial_t\bigg|^{\frac{1}{4}}(cf)\|_{L^2(dt dx)}
      \end{align*}
\end{lemma}

\begin{proof}

Recall
that if
\begin{align*}
B_z=
  \frac{1}{N} \int_{-\infty}^{\infty}c(t-s)e^{- i (t-s) \Delta}F_z(s, \cdot, \cdot) ds
\end{align*}
then
\begin{align*}
|\tilde B_z(\tau, \xi, \eta)|\lesssim \frac{1}{N} \frac{|\tilde F_z(\tau, \xi, \eta)|}{<\tau+|\xi|^2+|\eta|^2>}
\end{align*}

  Also,
  \begin{align*}
  &  \F \left(B_z(t, x+w, x-w)\right)(\tau, \xi)=\int \tilde B_z(\tau, \xi-\frac{\eta}{2},
      \xi+\frac{\eta}{2})e^{i w (\xi-\frac{\eta}{2})}e^{-i w (\xi+\frac{\eta}{2})} d \eta\\
       &  \F \left(P_{|\xi-\eta|<M}(B_z)(t, x+w, x-w)\right)(\tau, \xi)=\int \chi_{|\eta|<M} \tilde B_z(\tau, \xi-\frac{\eta}{2},
      \xi+\frac{\eta}{2})e^{-i w \eta} d \eta
      \end{align*}
  If  $F(t, x, y)=\delta(x-y-z)f(t, x+y)$,
  \begin{align*}
  &  \tilde F(\tau, \xi, \eta)=e^{i  \frac{\xi-\eta}{2} \cdot z}\F f(\tau, \frac{\xi+\eta}{2})\\
  &  \tilde F(\tau, \xi-\frac{\eta}{2},
      \xi+\frac{\eta}{2})e^{-i w \eta}
      =e^{-i w \cdot \eta} e^{-i  \eta \cdot z}\F (cf)(\tau,  \xi)
  \end{align*}
  Putting the above together,
  \begin{align*}
 &   \bigg|\F \bigg( \big(P_{|\xi-\eta| <M }B_z\big)(t, x+w, x-w)\bigg) \bigg|
    (\tau, \xi)\\
 & \lesssim \left( \int \chi_{|\eta| <M} \frac{1}{1+|\tau + 2|\xi|^2 +\frac{1}{2} |\eta|^2|} d \eta\right)|c * \tilde f|(\tau, \xi)
  \end{align*}

  and a direct calculation for the integral,
  \begin{align}
  \int \chi_{|\eta| <M} \frac{1}{1+|A + |\eta|^2|} d \eta \lesssim M \log M \label{directcal}
  \end{align}
(uniformly in $A$)  gives the pointwise result which, in turn, implies the $L^2$ results.
\end{proof}

Now we apply the above to the true $B$.

  \begin{proposition}
  Write
  \begin{align*}
& \frac{1}{N}  v_N(x-y)\Lambda(t, x, y)=\frac{1}{N} \int v_N(z)\delta(x-y-z)\Lambda(t, \frac{x+y+z}{2}, \frac{x+y-z}{2})\\
& =\frac{1}{N} \int v_N(z)\delta(x-y-z)f_z(t, x+y)\\
\end{align*}
and recall
\begin{align*}
  B= \frac{1}{N}
  \int_{-\infty}^{\infty}c(t-s)e^{- i (t-s) \Delta}c(s)  v_N(x-y)\Lambda(s, x, y) ds
\end{align*}
Then we have
\begin{align*}
&\sup_w\|<\nabla>^{\alpha}  \bigg(P_{|\xi-\eta|<M}c(t)B\bigg) (t, x+w, x-w)\|_{L^2(dt dx)}\\
&\lesssim \frac{M \log M}{N}
  \sup_w\|<\nabla>^{\alpha}  c(t)\Lambda\bigg) (t, x+w, x-w)\|_{L^2(dt dx)}\\
&  \lesssim \frac{M \log M}{N} \N(\Lambda) \lesssim N^{-\epsilon} \N(\Lambda)
\end{align*}
if $M=N^{\beta'}$.
Similarly,
\begin{align*}
&\sup_w\| \bigg|\partial_t\bigg|^{\frac{1}{4}}\bigg(P_{|\xi-\eta|<M}c(t)B\bigg) (t, x+w, x-w)\|_{L^2(dt dx)}\\
&\lesssim \frac{M \log M}{N}
  \sup_w\|\big|\partial_t\big|^{\frac{1}{4}} ( c(t)\Lambda)\bigg) (t, x+w, x-w)\|_{L^2(dt dx)}\\
&   \lesssim N^{-\epsilon} \N(\Lambda)
\end{align*}
  \end{proposition}

 \section{Estimates for $\|<\nabla_x>^{\alpha}<\nabla_y>^{\alpha}P_{|\xi-\eta|\lesssim N^{\beta'}}P_{|\xi+\eta|\lesssim N^{\beta'}}c(t)B\|_{L^{2}(dt) L^{6}( dx )L^2(dy)}$
 and $\|<\nabla_x>^{\alpha}<\nabla_y>^{\alpha}P_{|\xi-\eta|\lesssim N^{\beta'}}P_{|\xi+\eta|\lesssim N^{\beta'}}c(t)B\|_{L^{\infty}(dt) L^{2}( dx )L^2(dy)}$
 }

In this section we prove
\begin{proposition}
\begin{align*}
&\|<\nabla_x>^{\alpha}<\nabla_y>^{\alpha}P_{|\xi-\eta|\lesssim N^{\beta'}}P_{|\xi+\eta|\lesssim N^{\beta'}}c(t)B\|_{L^{2}(dt) L^{6}( dx )L^2(dy)}\\
&+ \|<\nabla_x>^{\alpha}<\nabla_y>^{\alpha}P_{|\xi-\eta|\lesssim N^{\beta'}}P_{|\xi+\eta|\lesssim N^{\beta'}}c(t)B\|_{L^{\infty}(dt) L^{2}( dx )L^2(dy)}\\
&\lesssim N^{-\epsilon}\left(\sup_z\||\partial_t|^{\frac{1}{4}}\Lambda(t, x, x+z)\|_{L^2(dt dx)}+
\sup_z\|<\nabla_x>^{\alpha}\Lambda(t, x, x+z)\|_{L^2(dt dx)}\right)
\end{align*}
\end{proposition}
\noindent thus completing the proof of Theorem \ref{mnthmlong}.
 Again, we prove a special case first
 \begin{lemma} \label{1/4} Fix $z$ and let
\begin{align*}
B_z(t, x, y)=
  \frac{1}{N} \int_{-\infty}^{\infty}c(t-s)e^{- i (t-s) \Delta}\delta(x-y-z)f_z(s, x+y) ds
\end{align*}
Then

\begin{align*}
&\| <\nabla_x >^{\alpha}  <\nabla_y >^{\alpha} P_{|\xi-\eta|\lesssim N^{\beta'}}
 P_{|\xi+\eta|\lesssim N^{\beta'}}
B_z\|_{L^{2}L^{6}L^2}\\
&+\| <\nabla_x >^{\alpha}  <\nabla_y >^{\alpha} P_{|\xi-\eta|\lesssim N^{\beta'}}
 P_{|\xi+\eta|\lesssim N^{\beta'}}
B_z\|_{L^{\infty}L^2L^2}\\
&\lesssim  N^{-\epsilon}\left( \|c(t)  <\nabla_x >^{\alpha} f_z\|_{L^2(dt dx)}
+\|\big|\partial_t\big|^{\frac{1}{4}}(cf_z)\|_{L^2(dt dx)}\right)
\end{align*}
\end{lemma}
\begin{proof}
Let
\begin{align*}
D_z= <\nabla_x >^{\alpha}  <\nabla_y >^{\alpha} P_{|\xi-\eta|\lesssim N^{\beta'}}
 P_{|\xi+\eta|\lesssim N^{\beta'}}
B_z
\end{align*}

We have a pointwise estimate for the Fourier transform
\begin{align*}
|\tilde D_z(\tau, \xi, \eta)| \lesssim \frac{N^{2 \beta' \alpha -1}}{<\tau + |\xi|^2 + |\eta|^2>} \chi(|\xi|+|\eta|\le C N^{\beta'})
 |\tilde f_z(\tau, \xi+\eta)|
\end{align*}
Fix $\beta'<\beta''<1$, with $\beta''-\beta'$ sufficiently small.

Break up $  D_z= D_{z, 0}+ D_{z, 1} + D_{z, 2}$ where
\begin{align*}
&|\tilde D_{z, 0}|\lesssim \\
&\bigg|\frac{N^{2 \beta' \alpha -1}\tilde f(\tau, \xi + \eta)}{<2\tau + |\xi+\eta|^2 +|\xi-\eta|^2>}\chi(|\xi|+|\eta|\le C N^{\beta'}, |\tau|> 10 C N^{2\beta''})\bigg|\\
&|\tilde D_{z, 1}|\lesssim \\
&\bigg|\frac{N^{2 \beta' \alpha -1}\tilde f(\tau, \xi + \eta)}{<2\tau + |\xi+\eta|^2 +|\xi-\eta|^2>}\chi_{|\xi-\eta|^2< 2 \left(2|\tau|+ |\xi+\eta|^2\right)}\chi(|\xi|+|\eta|\le C N^{\beta'}, |\tau|< 10 C N^{2\beta''}))\bigg|\\
&|\tilde D_{z, 2}|\\
&\lesssim \bigg|\frac{N^{2 \beta' \alpha -1}\tilde f(\tau, \xi + \eta)}{<2\tau + |\xi+\eta|^2 +|\xi-\eta|^2>}\chi_{|\xi-\eta|^2> 2 \left(2|\tau|+ |\xi+\eta|^2\right)}\chi(|\xi|+|\eta|\le C N^{\beta'}, |\tau|< 10 C N^{2\beta''})\bigg|
\end{align*}
All these terms can be estimated  in $X^{\frac{1}{2}+}$ or $X^{\frac{1}{4}+}$, by integrating with respect to $\xi-\eta$ first.

Explicitly, we write
\begin{align*}
  &\|
  \tilde D_{z, i}\|_{X^b}\lesssim \int|K(\tau, \xi + \eta,  \xi - \eta) f(\tau, \xi+\eta)|^2 d \tau d\xi d\eta\\
&=\int \left(\int |K(\tau, \xi + \eta,  \xi - \eta) |^2d (\xi-\eta)\right)|f(\tau, \xi+\eta)|^2 d \tau (d(\xi+\eta))
\end{align*}
(with a suitable choice of $K$) and compute the inner integral.

To show $D_{z, 0}\in X^{\frac{1}{2}+}$ for some number $\frac{1}{2}+$ slightly bigger that $\frac{1}{2}$,
\begin{align*}
\bigg(\int
&\bigg|\frac{<\tau+ |\xi|^2+|\eta|^2>^{\frac{1}{2}+} N^{2 \beta' \alpha -1}}{<2\tau + |\xi+\eta|^2 +|\xi-\eta|^2>}\chi(|\xi|+|\eta|\le 2 N^{\beta'}, |\tau|> 10 N^{2\beta''})\bigg|^2 d(\xi-\eta) \bigg)^{\frac{1}{2}}\\
&\lesssim N^{-\epsilon}|\tau|^{\frac{1}{4}}
\end{align*}
thus
\begin{align*}
\|D_{z, 0}\|_{L^{2}(dt)L^{6}(dx)L^2(dy)}+\|D_{z, 0}\|_{L^{\infty}(dt)L^{2}(dx)L^2(dy)}
&
\lesssim
\|D_{z, 0}\|_{X^{\frac{1}{2}+}} \lesssim N^{-\epsilon}\||\tau|^{\frac{1}{4}} \tilde f\|_{L^2(d \tau d \xi)}
\end{align*}

The argument for $\tilde D_{z, 1}$ is identical, we get (using the calculation \ref{directcal})
\begin{align*}
\|| D_{z, 1}\|_{X^{\frac{1}{2}+}} \lesssim N^{-\epsilon} \| |\left(|\tau| + |\xi|^2\right)^{\frac{1}{4}}\tilde f\|_{L^2(d\tau d\xi)}
\end{align*}

The main difference occurs in estimating $D_{z, 2}$. This term can only  be estimated (uniformly in $N$)in $X^{\frac{1}{4}+}$, not in
 $X^{\frac{1}{2}}$, but we have extra regularity in $x+y$ and $t$. We  have
\begin{align*}
&\int \bigg|\frac{N^{2 \beta' \alpha -1}<\tau+|\xi|^2+|\eta|^2>^{\frac{1}{4}+}}{<2\tau + |\xi+\eta|^2 +|\xi-\eta|^2>}\chi_{|\xi-\eta|^2> 2 \left(2|\tau|+ |\xi+\eta|^2\right)}\chi(|\xi|+|\eta|\le 2 N^{\beta'}, |\tau|< 10 N^{2\beta''})\bigg|^2 \\
&d(\xi-\eta)
\lesssim N^{-\epsilon}
\end{align*}
thus, using the "Sobolev at an angle" (see Lemma \ref{angle}) estimate and standard $X^{\frac{1}{4}+}$ estimates
\begin{align*}
&\|| D_{z, 2}\|_{L^{2}(dt)L^{6}(dx)L^2(dy)}\\
& \lesssim
\|<\nabla_x+\nabla_y>^{\alpha} D_{z, 2}\|_{L^{2+}(dt)L^{3-}(dx)L^2(dy)}\\
&\lesssim\|<\nabla_x+\nabla_y>^{\alpha} D_{z, 2}\|_{X^{\frac{1}{4}+}} \lesssim N^{-\epsilon} \| | <\xi>^{\alpha}\tilde f(\tau, \xi)\|_{L^2(d\tau d\xi)}
\end{align*}
To obtain $L^{\infty}(dt)$ estimates, we use
\begin{align*}
&\|| D_{z, 2}\|_{L^{\infty}(dt)L^2(dx)L^2(dy)}\\
&\lesssim\|<\partial_t>^{\frac{1}{4}} D_{z, 2}\|_{X^{\frac{1}{4}+}} \lesssim N^{-\epsilon} \| |<\tau>^{\frac{1}{4}}\tilde f(\tau, \xi)\|_{L^2(d\tau d\xi)}
\end{align*}

\end{proof}
The general case follows by taking  $f_z(x+y)=\Lambda(\frac{x+y+z}{2}, \frac{x+y-z}{2})
$
and writing
\begin{align*}
B= \int v_N(z) B_z dz
\end{align*}
We get
\begin{align*}
&\|<\nabla_x>^{\alpha}<\nabla_y>^{\alpha}P_{|\xi-\eta|\lesssim N^{\beta'}}P_{|\xi+\eta|\lesssim N^{\beta'}}c(t)B\|_{L^{2}(dt) L^{6}( dx )L^2(dy)}\\
&+ \|<\nabla_x>^{\alpha}<\nabla_y>^{\alpha}P_{|\xi-\eta|\lesssim N^{\beta'}}P_{|\xi+\eta|\lesssim N^{\beta'}}c(t)B\|_{L^{\infty}(dt) L^{2}( dx )L^2(dy)}\\
&\lesssim N^{-\epsilon}\left(\int |v_N|\right)\left(\sup_z\||\partial_t|^{\frac{1}{4}}\Lambda(t, x, x+z)\|_{L^2(dt dx)}+
\sup_z\||\nabla_x|^{\frac{1}{2}}\Lambda(t, x, x+z)\|_{L^2(dt dx)}\right)
\end{align*}

\end{proof}

Next, we record a form of Sobolev's inequality that has been used in the previous proof.
\begin{lemma} \label{angle}

  Assume the following Sobolev estimate holds in $\mathbb R^3$:\\
  $\|f\|_{L^q(\mathbb R^3)} \lesssim \|<\nabla>^{\alpha} f\|_{L^p(\mathbb R^3)}$.
  Then
\begin{align}
&\|F\|_{ L^{q}(dx) L^2(dy)}\notag
\lesssim
\|<\nabla_x+ \nabla_y>^{\alpha}F \|_{L^p(dx) L^2(dy)}\\
\end{align}
\end{lemma}
\begin{proof}

This follows from the following observation applied to
$K(x) = \F^{-1}(<\xi>^{-\alpha})$:
Let $K(x)$ be such that $\||K| * f\|_{L^q(\mathbb R^3)} \lesssim \|f\|_{L^p(\mathbb R^3)}$,
Let $L(x, y)=\delta(x-y)K(x+y)$. Then
\begin{align*}
  \|L*F\|_{L^q(dx)L^2(dy)} \lesssim \|F\|_{L^p( dx) L^2(dy)}
  \end{align*}
A direct calculation shows
\begin{align*}
L*F(x, y)=\int K(2(x-z))F(z, y-x+z) dz
\end{align*}
so
\begin{align*}
&\|L*F(x, \cdot)\|_{L^2(dy)} \lesssim \big(|K| * \|F\|_{L^2(dy)}\big)(x)\\
&\|L*F(x, \cdot)\|_{L^q(dx)L^2(dy)} \lesssim  \|F\|_{L^p(dx)L^2(dy)}
\end{align*}
\end{proof}

\newpage

\section{Appendix: Proof of Proposition \ref{quartertime}}

We present two proofs of Proposition \ref{quartertime}. The first one is a direct calculation and slight generalization. The second one uses Littlewood-Paley decompositions. Notice the similarity with Bourgain's improved Strichartz estimate \cite{Bourgain}.

\subsection{First proof}
Proposition \ref{quartertime} follows from the following lemma.
\begin{lemma} Let $\Lambda$, $F$ be as in Proposition \ref{quartertime}. Let $0< 2 \epsilon<\epsilon_1$,
\begin{align*}
&\widehat{\sigma}(\xi_{1},\xi_{2}):=\frac{\vert\xi_{1}\vert\vert\xi_{2}\vert}
{\vert\xi_{1}\vert^{2}+\vert\xi_{2}\vert^{2}}
\\
&0\le  \gamma_{1}<\frac{3-a-\epsilon_1}{4}\qquad ,\qquad 1\leq a <2
\\
&0\le\gamma_{2}<\frac{3-\epsilon_1}{4}\\
\end{align*}
Then
\begin{align}
&\sup_z \big\Vert(|\partial_{t}|^{a/4}+|\nabla_{x}|^{a/2})\Lm(t,x+z,x)\big\Vert_{L^{2}(dtdx)}
\leq
\\
&C_{1}
\left\Vert
\frac{\big(\vert\xi_{1}\vert\vert\xi_{2}\vert\big)^{\frac{1+a+\epsilon_{1}}{4}}
\widehat{\sigma}^{\gamma_1}\widehat{F}
(\tau,\xi_{2},\xi_{1})}
{\big(1+\vert\tau-\vert\xi_{1}\vert^{2}-\vert\xi_{2}\vert^{2}\vert\big)^{\frac{1+\epsilon}{2}}}\right\Vert^2_{L^{2}(d\xi_{2}d\xi_{1})}
\nonumber
\\
&+C_{2}\left\Vert
\frac{\big(\vert\xi_{1}\vert\vert\xi_{2}\vert\big)^{\frac{1+\epsilon_{1}}{4}}
\widehat{\sigma}^{\gamma_2}\widehat{F}
(\tau,\xi_{2},\xi_{1})}
{\big(1+\vert\tau-\vert\xi_{1}\vert^{2}-\vert\xi_{2}\vert^{2}\vert\big)^{\frac{2-a+\epsilon}{4}}}\right\Vert_{L^{2}(d\xi_{2}d\xi_{1})}
\label{b5-clps3}
\end{align}

\end{lemma}


\begin{proof}
First notice that we have to estimate the integral below,
\begin{align}
&\Big<\frac{\big(\vert\tau\vert^{a/4}+\vert\xi_{4}+\xi_{3}\vert^{a/2}\big)
\overline{\widehat{F}(\tau,\xi_{4},\xi_{3})}}{1+\vert\tau-\vert\xi_{4}\vert^{2}
-\vert\xi_{3}\vert^{2}\vert}
\frac{\big(\vert\tau\vert^{a/4}+\vert\xi_{2}+\xi_{1}\vert^{a/2}\big)
\widehat{F}(\tau,\xi_{2},\xi_{1})}{1+\vert\tau-\vert\xi_{2}\vert^{2}-\vert\xi_{1}\vert^{2}\vert}
\Big>_{\xi_{1}+\xi_{2}=\xi_{3}+\xi_{4}} \label{b5-quad1}\\
&:=\int <\cdots >\delta(\xi_{1}+\xi_{2}-\xi_{3}-\xi_{4})d \tau d \xi_1 \cdots d \xi_4 \notag
\end{align}
and we will employ the following inequality,
\begin{align}
\vert \tau\vert^{a/4}\leq C\left(\big\vert\tau-\vert\xi_{1}\vert^{2}
-\vert\xi_{2}\vert^{2}\big\vert^{a/4}+
\big(\vert\xi_{1}\vert^{2}+\vert\xi_{2}\vert^{2}
\big)^{a/4}\right)\ .
\label{b5-ineq1}
\end{align}
To compute the integral, let $\xi_{1+2}=\frac{\xi_1+\xi_2}{\sqrt 2}$, $\xi_{1-2}=\frac{\xi_1-\xi_2}{\sqrt 2}$ and similarly with $\xi_3$ and $\xi_4$. Then $\xi_{1+2}=\xi_{3+4}:=\xi$.
Next we set (in spherical coordinates)
\begin{align*}
&\xi_{1-2}:=\rho_{1}\omega_{1}\qquad ,\qquad \omega_{1}\in{\mathbb S}^{2}
\\
&\xi_{3-4}:=\rho_{2}\omega_{2}\qquad ,\qquad \omega_{2}\in{\mathbb S}^{2}
\end{align*}
We use Cauchy-Schwartz in the angular variables $(\omega_{1},\omega_{2})$
and for this purpose we introduce, (with $\beta$, $\gamma$ parameters to be chosen later)
\begin{align*}
\big(A_{\beta, \gamma}(\widehat{F})\big)^{2}:=
\int\limits_{{\mathbb S}^{2}} d\omega\left\{\big(\vert\xi-\rho\omega\vert\vert\xi+\rho\omega\vert \big)^{2\beta}
\left(\hat \sigma(\xi-\rho\omega, \xi+\rho\omega)\right)^{2 \gamma}
\vert\widehat{F}(\tau,\xi-\rho\omega,\xi+\rho\omega)\vert^{2}\right\}\ .
\end{align*}
Because of the inequality \ref{b5-ineq1} we obtain several terms one of which is
\begin{align}
&\big<\frac{(\vert\xi\vert^{2}+\rho^{2}_{2})^{a/4}I_{\beta}(\xi,\rho_{2})
A_{\beta}(\widehat{F}(\tau,\xi,\rho_{2}))}
{1+\vert\tau-\vert\xi\vert^{2}-\rho_{2}^{2}\vert}
\frac{(\vert\xi\vert^{2}+\rho^{2}_{1})^{a/4}I_{\beta}(\xi,\rho_{1})
A_{\beta}(\widehat{F}(\tau,\xi,\rho_{1}))}
{1+\vert\tau-\vert\xi\vert^{2}-\rho_{1}^{2}\vert}\big>_{d\mu}
\nonumber
\\
&{\rm where}\quad d\mu :=d\tau d\xi \rho_{1}^{2}\rho_{2}^{2}d\rho_{1}d\rho_{2}
\label{b5-int2}
\end{align}
and the integral $I_{\beta, \gamma}$ is,
\begin{align}
I_{\beta, \gamma}^{2}(\xi,\rho):=\int
\frac{d\omega}{(\vert\xi-\rho\omega\vert\vert\xi+\rho\omega\vert)^{2\beta}\left(\hat \sigma(\xi-\rho\omega, \xi+\rho\omega)\right)^{2 \gamma}}
\qquad 0<\beta <1\ .
\label{b5-Idef}
\end{align}
In addition, we have
\begin{align}
&\big<\frac{I_{\beta, \gamma}(\xi,\rho_{2})
A_{\beta, \gamma}(\widehat{F}(\tau,\xi,\rho_{2}))}
{(1+\vert\tau-\vert\xi\vert^{2}-\rho_{2}^{2}\vert)^{1-\frac{a}{4}}}
\frac{I_{\beta, \gamma}(\xi,\rho_{1})
A_{\beta, \gamma}(\widehat{F}(\tau,\xi,\rho_{1}))}
{(1+\vert\tau-\vert\xi\vert^{2}-\rho_{1}^{2}\vert)^{1-\frac{a}{4}}}\big>_{d\mu}
\label{b5-int4}
\end{align}
as well as cross terms which can be estimated in the same manner so we will ignore them.

The first important observation is the fact that we can estimate $I_{\beta, \gamma}$,
(we use the identity $(\vert\xi_{1}\vert\vert\xi_{2}\vert)^{2}=
(\xi_{1}\cdot\xi_{2})^{2}+\vert\xi_{1}\wedge\xi_{2}\vert^{2}$)
\begin{align*}
I^{2}_{\beta, \gamma}(\xi,\rho)&:=\int\limits_{{\mathbb S}^{2}}\frac{d\omega}{\big(\vert\xi-\rho\omega\vert\vert\xi+\rho\omega\vert\big)^{2\beta}\left(\hat \sigma(\xi-\rho\omega, \xi+\rho\omega)\right)^{2 \gamma}}\\
&=
\int\limits_{{\mathbb S}^{2}}\frac{\left(|\xi|^2+|\rho|^2\right)^{2 \gamma} d\omega}
{|\xi-\rho \omega|^{2 \beta+2\gamma}|\xi+\rho \omega|^{2 \beta+2\gamma}}\\
&=\left(|\xi|^2+|\rho|^2\right)^{2 \gamma}\int_{0}^{\pi}\frac{2\pi \sin(\phi)d\phi}
{\big((\vert\xi\vert^{2}-\rho^{2})^{2}+4\vert\xi\vert^{2}\rho^{2}\sin^{2}(\phi)\big)^{\beta+\gamma}}
\\
&=\frac{2\pi}{\big(\vert\xi\vert^{2}+\rho^{2}\big)^{2\beta-1}2\vert\xi\vert\rho}
\int\limits_{\vert z\vert\leq\frac{2\vert\xi\vert\rho}{\vert\xi\vert^{2}+\rho^{2}}}\frac{dz}{\big(1-z^{2}\big)^{\beta+\gamma}}
\sim \frac{C_{\beta, \gamma}}{\big(\vert\xi\vert^{2}+\rho^{2}\big)^{2\beta}}\ .
\end{align*}
provided $0 \le \beta + \gamma <1$
Thus we have the interesting fact,
\begin{align}
I_{\beta, \gamma}(\xi,\rho)\sim\frac{C_{\beta}}{(\vert\xi\vert^{2}+\rho^{2})^{\beta}}
\sim
\frac{C_{\beta}}{(\vert\xi_{1}\vert^{2}+\vert\xi_{2}\vert^{2})^{\beta}}
 \label{b5-Iest1}
\end{align}
due to averaging over the angular variable.

If we use Cauchy-Schwarz in the integrals \ref{b5-int2} and \ref{b5-int4} we end up estimating the
integral,
\begin{align*}
J^{2}(\tau ,\xi):=\int\limits_{{\mathbb R}^{+}}\frac{\rho^{2}d\rho}{
(1+\vert\tau -\vert\xi\vert^{2}-\rho^{2}\vert)^{1-\epsilon}
(\vert\xi\vert^{2}+\rho^{2})^{\frac{1+\epsilon_{1}}{2}}}
\sim \frac{1}{(1+\vert\tau-\vert\xi\vert^{2}\vert)^{\frac{\epsilon_{1}-2\epsilon}{2}}} \le C
\end{align*}
provided $\epsilon_1> 2\epsilon >0$.
Explicitly,  we get
\begin{align*}
\ref{b5-int2} =  \left\Vert \int
\frac{\rho^{2}d\rho(\vert\xi\vert^{2}+\rho^{2})^{a/4}I_{\beta, \gamma}(\xi,\rho)
A_{\beta, \gamma}(\widehat{F}(\tau,\xi,\rho))}
{1+\vert\tau-\vert\xi\vert^{2}-\rho^{2}\vert}\right\Vert^2_{L^2( d \tau d \xi)}
\end{align*}
Now we pick $\beta= \frac{1+a+ \epsilon_1}{4} $, $ \gamma=\gamma_1 <\frac{3-a-\epsilon_1}{4}$  ($0<2 \epsilon< \epsilon_1$) so that
\begin{align*}
\ref{b5-int2} &\le C  \left\Vert \int
\frac{\rho^{2}d\rho
A_{\beta, \gamma_1}(\widehat{F}(\tau,\xi,\rho))}
{(\vert\xi\vert^{2}+\rho^{2})^{\frac{1+\epsilon_1}{4}}\left(1+\vert\tau-\vert\xi\vert^{2}-\rho^{2}\vert\right)}\right\Vert^2_{L^2( d \tau d \xi)}\\
& \le C \left\Vert
\frac{\big(\vert\xi_{1}\vert\vert\xi_{2}\vert\big)^{\frac{1+\alpha+\epsilon_{1}}{4}}
\widehat{\sigma}^{\gamma_1}\widehat{F}
(\tau,\xi_{2},\xi_{1})}
{\big(1+\vert\tau-\vert\xi_{1}\vert^{2}-\vert\xi_{2}\vert^{2}\vert\big)^{\frac{1+\epsilon}{2}}}\right\Vert^2_{L^{2}(d\xi_{2}d\xi_{1})}
\end{align*}
As for \ref{b5-int4},

\begin{align*}
\ref{b5-int4} = \left\Vert \int
\frac{\rho^{2}d\rho I_{\beta, \gamma}(\xi,\rho)
A_{\beta, \gamma}(\widehat{F}(\tau,\xi,\rho))}
{(1+\vert\tau-\vert\xi\vert^{2}-\rho^{2}\vert)^{1-\frac{a}{4}}}
\right\Vert^2_{L^2( d \tau d \xi)}
\end{align*}
Here we pick $\beta=\frac{1+\epsilon_{1}}{4}$ and $0\le\gamma_2<\frac{3-\epsilon_{1}}{4}$
so that
\begin{align*}
\ref{b5-int4} & \le C \left\Vert \int
\frac{\rho^{2}d\rho
A_{\beta, \gamma_2}(\widehat{F}(\tau,\xi,\rho))}
{(\vert\xi\vert^{2}+\rho^{2})^{\frac{1+\epsilon_1}{4}}(1+\vert\tau-\vert\xi\vert^{2}-\rho^{2}\vert)^{1-\frac{a}{4}}}
\right\Vert^2_{L^2( d \tau d \xi)}\\
& \le C
\left\Vert
\frac{\big(\vert\xi_{1}\vert\vert\xi_{2}\vert\big)^{\frac{1+\epsilon_{1}}{4}}
\widehat{\sigma}^{\gamma_2}\widehat{F}
(\tau,\xi_{2},\xi_{1})}
{\big(1+\vert\tau-\vert\xi_{1}\vert^{2}-\vert\xi_{2}\vert^{2}\vert\big)^{\frac{2-a+\epsilon}{4}}}\right\Vert_{L^{2}(d\xi_{2}d\xi_{1})}
\end{align*}

\end{proof}

\subsection{Second proof}

\begin{lemma} Let $A=\{(\xi, \eta)\big| |\xi|\sim M,  |\eta|\sim N\}$ ($1< M \le N$) and let $\tilde F(\tau, \xi, \eta)$ supported in
$(\xi, \eta) \in A$ for every fixed $\tau$. Let $\epsilon>0$.

Let
\begin{align*}
|\tilde \Lambda(\tau, \xi, \eta) |\lesssim \frac{\tilde F(\tau, \xi, \eta)}{<\tau + |\xi|^2 + |\eta|^2>}
\end{align*}

Then, for all $\epsilon >0$,
\begin{align*}
&|\big|\partial_t\big|^{\frac{1}{4}}\Lambda(t, x, x)\big|\\
& \lesssim_{\epsilon} \left(\frac{M}{N}\right)^{\frac{1}{2}}\left(
\|<\nabla_x>^{\frac{1}{2}+\epsilon}<\nabla_y>^{\frac{1}{2}+\epsilon} F\|_{X^{-\left(\frac{1+\epsilon}{2}\right)}}+
\|<\nabla_x>^{\frac{1}{2}+\epsilon}<\nabla_y>^{\epsilon} F\|_{X^{-\left(\frac{1+\epsilon}{4}\right)}}\right)
\end{align*}

\end{lemma}

\begin{proof} The proof is based on the following calculation:
\begin{align*}
I=\int_{(\xi+ \frac{ \eta}{2},\xi-\frac{ \eta}{2}) \in A} \frac{1}{<\tau + |\xi+ \frac{ \eta}{2}|^2 + |\xi- \frac{ \eta}{2}|^2>^{1-\epsilon}<\xi+ \frac{ \eta}{2}>^{2 \epsilon}<\xi- \frac{ \eta}{2}>^{2 \epsilon}}d \eta \lesssim \frac{M^2}{N}
\end{align*}
If $|\tau|> N^2$, the integral is $\lesssim \frac{M^3}{N^2}$. So assume $|\tau| <N^2$. Then, writing $u=\tau + |\xi+ \frac{ \eta}{2}|^2 + |\xi- \frac{ \eta}{2}|^2$,
\begin{align*}
  I\lesssim \frac{1}{M^{2 \epsilon}N^{2 \epsilon}}\int_{|u|\lesssim N^2} \frac{1}{<u>^{1-\epsilon}}\int_{(\xi+ \frac{ \eta}{2},\xi-\frac{ \eta}{2}) \in A} \delta (\tau -u + |\xi+ \frac{ \eta}{2}|^2 + |\xi- \frac{ \eta}{2}|^2) d \eta du
  \end{align*}
and
\begin{align*}
\sup_{\tau, \xi} \int_{(\xi+ \frac{ \eta}{2},\xi-\frac{ \eta}{2}) \in A} \delta (\tau + |\xi+ \frac{ \eta}{2}|^2 + |\xi- \frac{ \eta}{2}|^2) d \eta\\
\sim \int_{(\xi+ \frac{ \eta}{2},\xi-\frac{ \eta}{2}) \in A} \frac{ d S}{|\eta|} \lesssim \frac{M^2}{N}
\end{align*}
where the integral is taken over the paraboloid $\tau + |\xi+ \frac{ \eta}{2}|^2 + |\xi- \frac{ \eta}{2}|^2=0$.

 If $N>>M$, then, in the region of integration,
 $|\eta| \ge c N$, and the area of a piece of the paraboloid in the ball
 $\{ \eta \big| |\xi+\eta| \le M\}$ is $\le C M^2$.
 If $M \sim N$, the integral is $\lesssim M$.

Continuing,
we write\\
 $|\tau|^{\frac{1}{4}} \le |\xi+ \frac{ \eta}{2}|^{\frac{1}{2}}+ |\xi- \frac{ \eta}{2}|^{\frac{1}{2}}
+ <\tau + |\xi+ \frac{ \eta}{2}|^2+|\xi- \frac{ \eta}{2}|^2>^{\frac{1}{4}}$
and have to estimate the corresponding terms in
\begin{align*}
|\tau|^{\frac{1}{4}}\big|\mathcal F (\Lambda(t, x, x))(\tau, \xi)\big| \lesssim \int \frac{|\tau|^{\frac{1}{4}}|\tilde F(\tau,
\xi+ \frac{ \eta}{2}, \xi- \frac{ \eta}{2})|}{<\tau + |\xi+ \frac{ \eta}{2}|^2 + |\xi- \frac{ \eta}{2}|^2>} d \eta
\end{align*}
We estimate, separately,

\begin{align*}
&\int \frac{|\xi- \frac{ \eta}{2}|^{\frac{1}{2}}|\tilde F(\tau,
\xi+ \frac{ \eta}{2}, \xi- \frac{ \eta}{2})|}{<\tau + |\xi+ \frac{ \eta}{2}|^2 + |\xi- \frac{ \eta}{2}|^2>} d \eta\\
& \le N^{\frac{1}{2}}\left(\int \frac{|\tilde F(\tau,
\xi+ \frac{ \eta}{2}, \xi- \frac{ \eta}{2})|^2 <\xi+ \frac{ \eta}{2}>^{2 \epsilon}<\xi- \frac{ \eta}{2}>^{2 \epsilon}}{<\tau + |\xi+ \frac{ \eta}{2}|^2 + |\xi- \frac{ \eta}{2}|^2>^{1+\epsilon}} d \eta\right)^{\frac{1}{2}}\\
&\left(\int_{(\xi+ \frac{ \eta}{2},\xi-\frac{ \eta}{2}) \in A} \frac{1}{<\tau + |\xi+ \frac{ \eta}{2}|^2 + |\xi- \frac{ \eta}{2}|^2>^{1-\epsilon}<\xi+ \frac{ \eta}{2}>^{2 \epsilon}<\xi- \frac{ \eta}{2}>^{2 \epsilon}} d \eta\right)^{\frac{1}{2}}\\
&\lesssim \left(\frac{M}{N}\right)^{\frac{1}{2}}
\left(\int \frac{|\tilde F(\tau,
\xi+ \frac{ \eta}{2}, \xi- \frac{ \eta}{2})|^2 <\xi+ \frac{ \eta}{2}>^{1+2 \epsilon}<\xi- \frac{ \eta}{2}>^{1+2 \epsilon}}{<\tau + |\xi+ \frac{ \eta}{2}|^2 + |\xi- \frac{ \eta}{2}|^2>^{1+\epsilon}} d \eta\right)^{\frac{1}{2}}
\end{align*}
Squaring and integrating $d \tau d\xi$, the above is\\
$ \lesssim
\left(\frac{M}{N}\right)^{\frac{1}{2}} \|<\nabla_x>^{\frac{1}{2}+\epsilon}<\nabla_y>^{\frac{1}{2}+\epsilon} F\|_{X^{-\left(\frac{1+\epsilon}{2}\right)}}$
The argument for
\begin{align*}
&\int \bigg|\int \frac{<\xi+ \frac{ \eta}{2}>^{\frac{1}{2}}|\tilde F(\tau,
\xi+ \frac{ \eta}{2}, \xi- \frac{ \eta}{2})|}{<\tau + |\xi+ \frac{ \eta}{2}|^2 + |\xi- \frac{ \eta}{2}|^2>} d \eta\bigg|
d \tau d\xi
\end{align*}
is easier, because $|\xi+\eta|\le |\xi-\eta|$ in $A$.
Finally,

\begin{align*}
&\int \bigg|\int \frac{<\tau + |\xi+ \frac{ \eta}{2}|^2 + |\xi- \frac{ \eta}{2}|^2>^{\frac{1}{4}}|\tilde F(\tau,
\xi+ \frac{ \eta}{2}, \xi- \frac{ \eta}{2})|}{<\tau + |\xi+ \frac{ \eta}{2}|^2 + |\xi- \frac{ \eta}{2}|^2>} d \eta\bigg|
d \tau d\xi\\
&\le\left(\int \frac{|\tilde F(\tau,
\xi+ \frac{ \eta}{2}, \xi- \frac{ \eta}{2})<\xi+ \frac{ \eta}{2}>^{ \epsilon}<\xi- \frac{ \eta}{2}>^{ \epsilon}|^2}{<\tau + |\xi+ \frac{ \eta}{2}|^2 + |\xi- \frac{ \eta}{2}|^2>^{\frac{1+2\epsilon}{2}}} d \eta\right)^{\frac{1}{2}}\\
&\left(\int \frac{1}{<\tau + |\xi+ \frac{ \eta}{2}|^2 + |\xi- \frac{ \eta}{2}|^2>^{1-\epsilon}<\xi+ \frac{ \eta}{2}>^{2 \epsilon}<\xi- \frac{ \eta}{2}>^{2 \epsilon}} d \eta\right)^{\frac{1}{2}}\\
&\lesssim \frac{M}{N^{1/2}} \left(\int \frac{|\tilde F(\tau,
\xi+ \frac{ \eta}{2}, \xi- \frac{ \eta}{2})<\xi+ \frac{ \eta}{2}>^{ \epsilon}<\xi- \frac{ \eta}{2}>^{ \epsilon}|^2}{<\tau + |\xi+ \frac{ \eta}{2}|^2 + |\xi- \frac{ \eta}{2}|^2>^{\frac{1+2\epsilon}{2}}} d \eta\right)^{\frac{1}{2}}\\
&\lesssim \left(\frac{M}{N}\right)^{\frac{1}{2} }\left(\int \frac{||\xi+ \frac{ \eta}{2}|^{\frac{1}{2} }\tilde F(\tau,
\xi+ \frac{ \eta}{2}, \xi- \frac{ \eta}{2})<\xi+ \frac{ \eta}{2}>^{ \epsilon}<\xi- \frac{ \eta}{2}>^{ \epsilon}|^2}{<\tau + |\xi+ \frac{ \eta}{2}|^2 + |\xi- \frac{ \eta}{2}|^2>^{\frac{1+2\epsilon}{2}}} d \eta\right)^{\frac{1}{2}}
\end{align*}

\end{proof}

Using this, we can prove
\begin{proposition} Let $\alpha >\frac{1}{2}$. Then
\begin{align*}
\sup_z \|\big|\partial_t\big|^{\frac{1}{4}}\left(\Lambda(t, x, x+z)\right)\|_{L^2(dt dx)} \lesssim
\|<\nabla_x>^{\alpha}<\nabla_y>^{\alpha} F\|_{X^{-\frac{1}{2}-}}\\+
\|<\nabla_x>^{\alpha}<\nabla_y>^{\alpha-\frac{1}{2}} F\|_{X^{-\frac{1}{4}-}}
\end{align*}
\end{proposition}

Recall we chose  $\varphi\in C_0^{\infty}(-2, 2)$ to be identically $1$ on $[-1, 1]$, and let $\beta(\xi)=\varphi(\xi)-\varphi(2\xi)$, so that
$\phi(\xi)+ \sum_{k=1}^{\infty} \beta(\frac{\xi}{2^k})=1$. For function $f(x)$, $\Lambda(x, y)$, and $k \ge 1$, define the projections
$P_{2^k}(f)$, $
P^1_{2^k}P^2_{2^l}(\Lambda)$ by localizing in Fourier space

\begin{align*}
&\mathcal F (P_{2^k}(f))(\xi)=\beta\left(\frac{|\xi|}{2^k}\right)\hat f(\xi)\\
&\mathcal F (P^1_{2^k}P^2_{2^l}(\Lambda))(\xi, \eta)=\beta\left(\frac{|\xi|}{2^k}\right)\beta\left(\frac{|\eta|}{2^l}\right)\hat \Lambda(\xi, \eta)
\end{align*}
We choose to include all low frequencies in $P_{2^0}$:
\begin{align*}
&\mathcal F (P_1(f))(\xi)=\phi(|\xi|)\hat f(\xi)\\
&\mathcal F (P^1_1P^2_1(\Lambda)(\xi, \eta)=\phi(|\xi|)\phi(|\eta|)\hat \Lambda(\xi, \eta)
\end{align*}

We have the equivalent of the standard Littlewood-Paley product decomposition (in space variables):
\begin{align*}
&P_{2^i}\left(\Lambda(t, x, x)\right)\\
&= \left(\sum_{2^i<<2^j\sim2^k }+
\sum_{2^i \sim 2^j\sim2^k }
+\sum_{2^i\sim2^k >> 2^j} +
\sum_{2^i\sim2^j >> 2^k}
\right)
P_{2^i}\left(P^1_{2^j} P^2_{2^k} \Lambda(t, \cdot, \cdot)\right)
\end{align*}
and (taking $\alpha - \frac{1}{2} > 2 \epsilon>0$)
\begin{align*}
&\|\big|\partial_t\big|^{\frac{1}{4}}P_{2^i}(\Lambda(t, x, x))\|_{L^2(dt dx)}\\
&\le
\left(\sum_{2^i<<2^j\sim2^k }+
\sum_{2^j\sim2^k \sim 2^i}
+\sum_{2^i\sim2^k >> 2^j} +
\sum_{2^i\sim2^j >> 2^k}
\right)
\|\big|\partial_t\big|^{\frac{1}{4}}P_{2^i}\left(\left(P^1_{2^j} P^2_{2^k} \Lambda\right)(t, x, x)\right)\|_{L^2(dt dx)}\\
&\lesssim
\left(\sum_{2^i<<2^j\sim2^k }+
\sum_{2^i\sim2^j \sim 2^k}\right)
\bigg(
\|<\nabla_x>^{\frac{1}{2}+\epsilon}<\nabla_y>^{\frac{1}{2}+\epsilon}P^1_{2^j} P^2_{2^k} F\|_{X^{-\frac{1+\epsilon}{2}}}\\
&+
\|<\nabla_x>^{\frac{1}{2}+\epsilon}<\nabla_y>^{\epsilon}P^1_{2^j} P^2_{2^k} F\|_{X^{-\frac{1+\epsilon}{4}}}\bigg)\\
&+\sum_{2^i\sim2^k >> 2^j}  2^{\frac{j-k}{2}} \bigg(
\|<\nabla_x>^{\frac{1}{2}+\epsilon}<\nabla_y>^{\frac{1}{2} + \epsilon}P^1_{2^j} P^2_{2^k} F\|_{X^{-\frac{1+\epsilon}{2}}}\\+
&\|<\nabla_x>^{\frac{1+\epsilon}{2}}<\nabla_y>^{\epsilon} P^1_{2^j} P^2_{2^k}F\|_{X^{-\frac{1+\epsilon}{4}}}\bigg)\\
&+\sum_{2^i\sim2^j >> 2^k} 2^{\frac{k-j}{2}}\bigg(
\|<\nabla_x>^{\frac{1+\epsilon}{2}}<\nabla_y>^{\frac{1+\epsilon}{2}}P^1_{2^j} P^2_{2^k} F\|_{X^{-\frac{1+\epsilon}{2}}}\\
&+
\|<\nabla_x>^{\frac{1+\epsilon}{2}}<\nabla_y>^{\epsilon} P^1_{2^j} P^2_{2^k}F\|_{X^{-\frac{1+\epsilon}{4}}}\bigg)\\
&\lesssim
\left(\sum_{2^i<<2^j\sim2^k }+
\sum_{2^i\sim2^j \sim 2^k}\right)
2^{-j(\alpha-\frac{1}{2}-\epsilon)}\bigg(
\|<\nabla_x>^{\alpha}<\nabla_y>^{\alpha}P^1_{2^j} P^2_{2^k} F\|_{X^{-\frac{1+\epsilon}{2}}}\\
&+
\|<\nabla_x>^{\alpha}<\nabla_y>^{\alpha-\frac{1}{2}}P^1_{2^j} P^2_{2^k} F\|_{X^{-\frac{1+\epsilon}{4}}}\bigg)\\
&+\sum_{2^i\sim2^k >> 2^j} 2^{- i(\alpha-\frac{1}{2}-\epsilon)} 2^{\frac{j-k}{2}} \bigg(
\|<\nabla_x>^{\alpha}<\nabla_y>^{\alpha}P^1_{2^j} P^2_{2^k} F\|_{X^{-\frac{1+\epsilon}{2}}}\\
&+
\|<\nabla_x>^{\alpha}<\nabla_y>^{\alpha-\frac{1}{2}} P^1_{2^j} P^2_{2^k}F\|_{X^{-\frac{1+\epsilon}{4}}}\bigg)\\
&+\sum_{2^i\sim2^j >> 2^k} 2^{- i(\alpha-\frac{1}{2}-\epsilon)} 2^{\frac{k-j}{2}}\bigg(
\|<\nabla_x>^{\alpha}<\nabla_y>^{\alpha}P^1_{2^j} P^2_{2^k} F\|_{X^{-\frac{1+\epsilon}{2}}}\\
&+
\|<\nabla_x>^{\alpha}<\nabla_y>^{\alpha-\frac{1}{2}} P^1_{2^j} P^2_{2^k}F\|_{X^{-\frac{1+\epsilon}{4}}}\bigg)\\
&\lesssim 2^{- i(\alpha-\frac{1}{2}-\epsilon)}\left(
\|<\nabla_x>^{\alpha}<\nabla_y>^{\alpha} F\|_{X^{-\frac{1+\epsilon}{2}}}+
\|<\nabla_x>^{\alpha}<\nabla_y>^{\alpha-\frac{1}{2}} F\|_{X^{-\frac{1+\epsilon}{4}}}\right)
\end{align*}
Now we sum over $i$ to get the desired result.


\begin{thebibliography}{24}

\bibitem{BBCFS} V. Bach, S. Breteaux, T. Chen, J. Fr\"ohlich, I. M. Sigal: {The time-dependent Hartree-Fock-Bogoliubov
equation for Bosons.}  arXiv:1602.05171



\bibitem{BSS} N. Benedikter, J. Sok, J. P. Solovej : The Dirac-Frenkel Principle for Reduced Density Matrices, and the Bogoliubov-de Gennes Equations, Ann. Henri Poincar\'e 19 (2018), 1167-1214




\bibitem{B-O-S}
Benedikter, N., de Oliveira, G., Schlein B. :{
Quantitative Derivation of the Gross-Pitaevskii Equation},
Communications on Pure and Applied Mathematics
Volume 68, Issue 8, pages 1399 - 1482, 2015

\bibitem{BPS} Benedikter, N., Porta, M. Schlein, B.
: {Mean Field Evolution of Fermionic systems}
Communications in Mathematical Physics, Volume 331, Issue 3, pp 1087–1131 (2014)

\bibitem{BCS}  Boccato, C. , Cenatiempo, N., Schlein, B. :{
Quantum many-body fluctuations around nonlinear
Schr ¨odinger dynamics} Annales Henri Poincaré, Volume 18, Issue 1, pp 113–191 (2017)

\bibitem{Bog47} Bogoliubov, N.~N.:{On the theory of superfluidity}, Journal of Physics, Vol XI, (1947)

\bibitem{Bourgain} J. Bourgain, Refinements of Strichartz inequality and applications to 2D-NLS with critical nonlinearity,
IMRN 1998, N5, 253-283.

\bibitem{XC1} Chen, X. : {Second order corrections to mean field evolution for weakly interacting bosons in the case of 3 body interactions}, Arch. Rational Mech. Anal. \textbf{203} (2012), 455-497


\bibitem{C-H1} Chen, X., Holmer, J.,
On the Klainerman-Machedon Conjecture of the Quantum BBGKY
Hierarchy with Self-interaction
, to appear in Journal of the European Mathematical
Society.



\bibitem{C-H2} X. Chen, J. Holmer,  Correlation structures, Many-body Scattering Processes and the Deriva-
tion of the Gross-Pitaevskii Hierarchy. Int. Math. Res. Not. 2016 (10), 3051-3110.


\bibitem{jacky1} Chong, Jacky  Xia Wei. : Uniform in N global well-posedness of the
time-dependent Hartree-Fock-Bogoliubov equations
in $\mathbb R^{1+1}$, to appear in Lett Math Phys
https://doi.org/10.1007/s11005-018-1078-8




\bibitem{jacky2}Chong, Jacky Xia Wei: Dynamical Hartree-Fock-Bogoliubov Approximation of Interacting Bosons, arXiv:1711.00610


\bibitem{CHP} Chen, T., Hong, Y., Pavlovic, N. : Global well posedness for the NLS system for infinitely many Fermions, Arch. Rat. Mechanics and Analysis,  \textbf{ 224}, 91--123  (2017)


   \bibitem{EMS} L. Erd\"̋os, A. Michelangeli, B. Schlein: Dynamical Formation of Correlations in a Bose-Einstein
Condensate. Comm. Math. Phys. 289 (2009), 1171–1210.




     \bibitem{G-V}Ginibre, J., Velo, G.:
{The classical field limit of scattering theory for non-relativistic many-boson systems, I and II}.
Comm. Math. Phys. \textbf{66}, 37--76 (1979) and \textbf{68}, 45--68 (1979)

\bibitem{GMM1}Grillakis, M. Machedon. M, Margetis, D.: {Second-order corrections to mean field evolution of weakly interacting Bosons. I.}
Comm. Math. Phys. \textbf{294}, 273--301 (2010)

\bibitem{GMM2}Grillakis, M. Machedon. M, Margetis, D.: {Second-order corrections to mean field evolution of weakly interacting Bosons. II.}
Adv. in Math. \textbf{228}, 1788-1815 (2011)

\bibitem{GM1} Grillakis, M. Machedon. M,: {Pair excitations and the mean field approximation of interacting Bosons, I.}
 Communications in Mathematical Physics, Volume 324, Issue 2, pp 601-636,
 Oct (2013)

\bibitem{GM2} Grillakis, M. Machedon. M, : {Beyond mean field: On the role of pair excitations in the evolution of condensates},
Journal of fixed point theory and applications, vol 14, no 1, (The Yvonne Choquet-Bruhat Festschrift),  pp 91-111,  2013.

\bibitem{GM3}  M. Grillakis, M. Machedon: Pair excitations and the mean field approximation of interacting bosons, II, Communications in PDE, Vol 42, No 1, 24--67 (2017)

\bibitem{hepp}Hepp, K.: {The classical limit for quantum mechanical correlation functions}. Comm. Math. Phys. \textbf{35}, 265--277 (1974)

\bibitem{KM0} Klainerman, S., Machedon, M.: {Space-time estimates for null forms and the local existence theorem}.
Comm. Pure Appl. Math. \textbf{46}, 1221--1268 (1993)

\bibitem{K-MMM}
 Klainerman, S.,  Machedon,M. { On the uniqueness of solutions to the
    Gross-Pitaevskii
hierarchy}.  Comm. Math.\ Phys.\ \textbf{279}, 169-185 (2008)

\bibitem{K-P} Knowles, A., Pickl, P. : { Mean-field dynamics: Singular potentials and rate of convergence}, Comm. Math Phys \textbf{298}, 101--138 (2010).

\bibitem{elif2} Kuz, E. {Exact evolution versus mean field with second order correction for Bosons interacting via short-range two-body potential},  Differential and Integral Equations,
    Volume 30, Number 7/8 (2017), 587-630.





\bibitem{Rod-S}
Rodnianski, I., Schlein, B.: {Quantum fluctuations and rate of convergence towards mean field dynamics}.
Comm. Math. Phys. \textbf{291}(2), 31--61 (2009)


\bibitem{BNNS} C. Brennecke, P. T. Nam, M. Napiorkowski and B. Schlein : Fluctuations of N-particle quantum dynamics around the non-linear Schr\"odinger equation
arXiv:1710.09743v1




\bibitem{shale}
Shale, D.: {Linear symmetries of free Boson fields}, Trans. Amer. Math. Soc. \textbf{103}(1), 149--167 (1962)

\bibitem{Stein} E. M. Stein: {\it Harmonic Analysis, Real Variable Methods Orthogonality and Oscillatory Integrals}, Princeton University Press, (1993)

\bibitem{taobook}
T. Tao, {\it Nonlinear dispersive equations: local and global analysis, CBMS regional series in mathematics}, (2006)

\bibitem{Tataru} Tataru, D. : The $X^s_{\theta}$ spaces and unique continuation for solutions to the semilinear wave equation, Comm. PDE, 21, (5-6), 841-887 (1996)

\end{thebibliography}
\end{document}